\documentclass[12pt]{amsart}
\usepackage{amssymb,amsthm}
\usepackage{amsmath}
\usepackage[foot]{amsaddr}
\usepackage{tikz}
\usetikzlibrary{arrows}
\usepackage{graphicx}
\usepackage{subcaption}
\usepackage[shortlabels]{enumitem}
\usetikzlibrary{arrows}
\usepackage{float}
\usepackage{graphicx}
\usepackage{subcaption}
\usepackage{hyperref}
\usepackage{pgfplots}

\graphicspath{{Ellipse_figs/}}
\usetikzlibrary{decorations.pathmorphing}
\tikzset{snake it/.style={decorate, decoration=snake}}

\makeatletter
\newcommand*{\rom}[1]{\expandafter\@slowromancap\romannumeral #1@}
\makeatother

\numberwithin{equation}{section}

\theoremstyle{plain}
\newtheorem{theorem}{Theorem}
\numberwithin{theorem}{section}
\newtheorem{proposition}[theorem]{Proposition}

\newtheorem{corollary}[theorem]{Corollary}
\theoremstyle{definition}
\newtheorem{definition}[theorem]{Definition}
\theoremstyle{remark}
\newtheorem{remark}[theorem]{Remark}
\newtheorem{example}[theorem]{Example}
\theoremstyle{remark}

\theoremstyle{remark}

\newcommand{\Ceps}{C_{\epsilon}}

\newcommand{\eps}{\epsilon}
\newcommand{\smo}{\setminus \mathbf{0}}
\newcommand{\zero}{\mathbf{0}}

\newcommand{\norm}[1]{\left\lVert#1\right\rVert}      
\newcommand{\abs}[1]{\left|#1\right|}                 
\newcommand{\paren}[1]{\left(#1\right)}               
\newcommand{\bparen}[1]{\left[#1\right]}               
\newcommand{\sparen}[1]{\left\{#1\right\}}      

\renewcommand{\d}{\,\mathrm{d}}  
\newcommand{\sym}{\mathrm{Sym}(n)}
\newcommand{\dd}{\mathrm{d}}  

\newcommand{\Cc}{\mathcal{C}}

\newcommand{\Dc}{\mathcal{D}}
\newcommand{\Ec}{\mathcal{E}}
\newcommand{\Fc}{\mathcal{F}}

\newcommand{\Nc}{\mathcal{N}}

\newcommand{\Sc}{\mathcal{S}}

\newcommand{\WF}{\mathrm{WF}}                         
\newcommand{\wf}{\mathrm{WF}}                         

\newcommand{\vb}{\mathbf{b}}

\newcommand{\inv}{^{-1}}

\newcommand{\partyf}[2]{\frac{\partial #2}{\partial y_{#1}}}

\newcommand{\vu}{{\mathbf{u}}}
\newcommand{\vv}{{\mathbf{v}}}

\newcommand{\bpm}{\begin{pmatrix}}
\newcommand{\epm}{\end{pmatrix}}

\newcommand{\vx}{{\mathbf{x}}}

\newcommand{\vso}{\mathbf{s}_0}

\newcommand{\vy}{{\mathbf{y}}}
\newcommand{\vz}{{\mathbf{z}}}
\newcommand{\vs}{\mathbf{s}}

\newcommand{\vxi}{{\boldsymbol{\xi}}}
\newcommand{\vxio}{{\boldsymbol{\xi}_0}}

\newcommand{\vsig}{{\boldsymbol{\sigma}}} 

\newcommand{\phio}{\phi_0}

\newcommand{\hxo}{\hat{x}_1} %
\newcommand{\hxt}{\hat{x}_2} %
\newcommand{\hphi}{\hat{\phi}}

\newcommand{\fnshphi}{\text{\footnotesize{$\hphi$}}}

\newcommand{\rr}{{{\mathbb R}}}

\newcommand{\rn}{{{\mathbb R}^n}}

\newcommand{\dR}{\dot{\mathbb{R}}}
\newcommand{\drn}{{\dot{{\mathbb R}^n}}}
\newcommand{\rnmo}{{\mathbb{R}^{n-1}}}

\newcommand{\st}{\hskip 0.3mm : \hskip 0.3mm}

\newcommand{\be}{\begin{equation}}
\newcommand{\ee}{\end{equation}}
\newcommand{\bea}{\begin{eqnarray}}
\newcommand{\eea}{\end{eqnarray}}
\newcommand{\bean}{\begin{eqnarray*}}
\newcommand{\eean}{\end{eqnarray*}}

\newcommand{\bel}[1]{\begin{equation}\label{#1}}

\newcommand{\eel}[1]{{\label{#1}\end{equation}}}

\DeclareMathOperator*{\argmin}{arg\,min}
\DeclareMathOperator{\sinc}{sinc}

\usepackage{fullpage}

\usepackage{kbordermatrix}
\renewcommand{\kbldelim}{(}
\renewcommand{\kbrdelim}{)}

\usepackage{datetime}
\usetikzlibrary{quotes, angles}

\newcommand\irregularcircle[2]{
  \pgfextra {\pgfmathsetmacro\len{(#1)+rand*(#2)}}
  +(0:\len pt)
  \foreach \a in {10,20,...,350}{
    \pgfextra {\pgfmathsetmacro\len{(#1)+rand*(#2)}}
    -- +(\a:\len pt)
  } -- cycle
}

\title[short]{Ellipsoidal and hyperbolic Radon transforms; microlocal properties and injectivity\\{\footnotesize\ddmmyyyydate\today~\currenttime}}
\author{James W. Webber\textsuperscript{$\dagger$}}
\author{Sean Holman\textsuperscript{$\ddagger$}}
\author{Eric Todd Quinto*}
\address[James W. Webber (corresponding author)]{Department of Oncology and Gynecology, Brigham and Women's Hospital, 221 Longwood Ave. Boston, MA 02115}
\address[Sean Holman]{Department of Mathematics, The University of Manchester, Alan Turing Building, Oxford Road, Manchester M13 9PY}
\address[Eric Todd Quinto]{Department of Mathematics, Tufts
University, 177 College Ave, Medford, MA 02155}
\email[A1,A2]{jwebber5@bwh.harvard.edu\textsuperscript{$\dagger$}, sean.holman@manchester.ac.uk\textsuperscript{$\ddagger$}}
\email[A3]{todd.quinto@tufts.edu*}

\providecommand{\keywords}[1]
{
  \small	
  \textbf{\textit{Keywords---}} #1
}

\begin{document}

\begin{abstract}
We present novel microlocal and injectivity analyses of ellipsoid and
hyperboloid Radon transforms. We introduce a new Radon transform,
$R$, which defines the integrals of a compactly supported $L^2$
function, $f$, over ellipsoids and hyperboloids with centers on a
smooth connected surface, $S$. 
$R$ is shown to be a Fourier Integral Operator (FIO) and in our main
theorem we prove that $R$ satisfies the Bolker condition if  the
support of $f$ is connected and not intersected by any plane tangent
to $S$. Under
certain conditions, this is an equivalence.
We give examples where our theory can be applied.   
Focusing specifically on a cylindrical geometry of interest in Ultrasound Reflection Tomography (URT), we prove injectivity results and investigate the visible singularities. In addition, we present example reconstructions of image phantoms in two-dimensions, and validate our microlocal theory.
\end{abstract}


\maketitle

\section{Introduction}

In this paper, we introduce a novel Radon transform, $R$, which defines the integrals of compactly supported $L^2$ functions in $\mathbb{R}^n$ over ellipsoid, two-sheeted hyperboloid, and elliptic hyperboloid surfaces, with centers on a smooth, $(n-1)$-dimensional hypersurface, which we denote by $S$. $R$ has applications in many imaging fields, such as Ultrasound Reflection Tomography (URT),  Photoacoustic Tomography (PAT), ground penetrating radar, and Synthetic Aperture Radar (SAR). We present a novel microlocal and injectivity analysis of $R$, and determine the singularities (image edges) detected by $R$ in examples of interest in URT.

The literature considers microlocal and injectivity analysis of
spherical and ellipsoidal Radon transforms
\cite{gouia2012approximate,krishnan2012microlocal,SU:SAR2013,Caday:SAR,Rubin-sonar,CokerTewfik,ABKQ2013,HomanZhou,Nguyen-Pham,Co1963,haltmeier2017spherical,rubin2002inversion,kunyansky2007explicit,grathwohl2020imaging,moon2016inversion,
andersson1988determination}. Analytic uniqueness is considered in
\cite{HomanZhou}. In \cite{Nguyen-Pham}, the authors consider a Radon transform, $\mathcal{R}$, which defines the integrals of an $n$-D function over $(n-1)$-dimensional spheres with centers on a smooth, strongly convex hypersurface, denoted by $\mathcal{S}$ (using the notation of \cite{Nguyen-Pham}). The authors show that $\mathcal{R}$ is a Fourier Integral Operator (FIO) with left projection that drops rank on planes tangent to $\mathcal{S}$. More precisely, the left and right projections of $\mathcal{R}$ are shown to be Whitney folds. This means that there are artifacts in filtered backprojection type reconstructions from $\mathcal{R}f$ data which are reflections in hyperplanes tangent to $\mathcal{S}$.

In \cite{ABKQ2013}, the authors present a microlocal analysis of an elliptic Radon transform, $\mathcal{R}$ (to adopt the notation of \cite{ABKQ2013}), of interest in two-dimensional URT. The authors consider a scanning modality, whereby a single emitter-receiver pair, kept a fixed distance apart, are rotated about the origin on lines tangent to the unit circle. The reflectivity function, which is the reconstruction target in URT, is supported on the interior of the unit circle. $\mathcal{R}f$ has two degrees of freedom, which are the major diameter of the ellipse, and the position of ellipse center, which lies on the unit circle and follows the emitter-receiver rotation. The authors prove that $\mathcal{R}$ is an elliptic FIO with conical relation, $\mathcal{C}$, which satisfies the Bolker condition. After which, it is shown that the normal operator of $\mathcal{R}$ is an elliptic Pseudodifferential Operator (PDO), order $-1$, and thus the inverse of $\mathcal{R}$ is stable on Sobolev scale $\frac{1}{2}$. 

In \cite{haltmeier2017spherical}, the authors consider a spherical Radon transform. The spheres of integration have centers restricted to  cylindrical hypersurfaces of the form $\Gamma \times \mathbb{R}^{m}$, where $\Gamma$ is a hypersurface in $\mathbb{R}^n$. The authors present a general methodology for inverting spherical Radon transforms with center set $\Gamma \times \mathbb{R}^{m}$. Specifically, the authors show that if an inversion formula is known for the center set $\Gamma$, then this can be extended to $\Gamma \times \mathbb{R}^{m}$. They apply the theory of \cite{andersson1988determination}, which provides inversion formulae for the spherical Radon transform with a flat plane center set, to derive inversion formulae for elliptic and circular cylinder center sets. Numerical results are also provided when the set of sphere centers is an elliptic cylinder, and the authors present simulated reconstructions of image phantoms from spherical integral data using the proposed formulae. A blurring effect is observed near sharp discontinuities in the image reconstructions, indicating that all singularities are not well resolved when the center set is an elliptic cylinder.

In our work, we introduce a novel Radon transform, denoted by
$R$, which defines the integrals over ellipsoids and hyperboloids with
centers on a smooth, connected surface, $S$. We show that $R$ is an
FIO. Our central theorem proves that $R$ satisfies the Bolker
condition if and only if $\text{supp}(f)$ is not intersected by any
hyperplane tangent to $S$. The Bolker condition is important as
it relates to image artifacts in filtered backprojection type
reconstructions from Radon transform data, specifically to artifacts
which are additional (unwanted) singularities in the reconstruction
that are not in the object. Such artifacts are also often observed
using iterative solvers and algebraic reconstruction techniques
\cite{WebberQuinto2020I}. If the Bolker condition is satisfied, this
implies reconstruction stability, and unwanted microlocal
singularities are eliminated. Conversely, if the Bolker condition
fails, the capacity for artifacts is amplified. 

The calculations which determine satisfaction of Bolker shed light on
the nature of the image artifacts (should they exist) if Bolker fails
and can be used to predict artifact location and to help suppress
artifacts \cite{felea2013microlocal,webberholman}.

In a similar vein to \cite{Nguyen-Pham}, the
left projection of $R$ is shown to drop rank on planes which are
tangent to $S$ and we discover ``mirror point" type artifacts which
occur on opposite sides of planes tangent to $S$. Specifically, if the tangent planes to $S$ do not intersect $\text{supp}(f)$, then we show
that the artifacts are constrained to lie outside of
$\text{supp}(f)$, and thus the Bolker condition holds. This is one of
the central ideas of our main theorem. In \cite{Nguyen-Pham}, the
surfaces of integration are spheres, which are symmetric about any
plane through their center. This causes the reflection artifacts
discovered in \cite{Nguyen-Pham}. However, ellipsoids and hyperboloids do not share such symmetries, and thus the artifacts we discover are not reflections through planes tangent to $S$, as in \cite{Nguyen-Pham}, but can be understood as a ``perturbed" or ``distorted" reflection. See Section  \ref{sect:visible}, for a more detailed discussion on mirror point artifacts. The microlocal theory we present here is a generalization of the work of \cite{Nguyen-Pham}, to ellipsoid and hyperboloid integration surfaces.

After establishing our central microlocal theorems, we present a number of examples where our theory can be applied, some of which are relevant to URT. We focus on a cylindrical scanning geometry in $\mathbb{R}^3$, of interest in URT, and prove injectivity results. Specifically, we prove that any $L^2$ function, $f$, compactly supported on the interior of a unit cylinder in $\mathbb{R}^3$, can be reconstructed uniquely from its integrals over spheroids with centers on the unit cylinder. A unit cylinder in $\mathbb{R}^3$ is a special case of the more general cylindrical hypersurfaces considered in \cite{haltmeier2017spherical}. The authors of \cite{haltmeier2017spherical} consider spherical integral surfaces, whereas we consider, more general, spheroid integral surfaces. Our injectivity results hold for compactly supported $L^2$ functions, which advances the theory of \cite{haltmeier2017spherical}, as their inversion formulae apply only to smooth functions of compact support. In addition, we show, using Volterra integral equation theory \cite{Q1983-rotation}, that, with limited spheroid radii, one can reconstruct $f$ on cylindrical tubes (or ``layers") which are subsets of the unit cylinder interior. Limited sphere radii are not considered in \cite{haltmeier2017spherical}. We aim to address limited spheroid and sphere radii in this work. 

The remainder of this paper is organized as follows. In section \ref{sect:defns}, we give some definitions from microlocal analysis that will be used in our theorems. In section \ref{sect:theory}, we define our generalized Radon transform and prove our main microlocal theorems, and follow up with some examples in section \ref{sect:examples}. In section \ref{cylinder:sect}, we investigate a cylindrical scanning geometry with applications in URT, and prove our main injectivity theorems. We also discuss in detail the visible singularities and show how the wavefront coverage varies with emitter/receiver discretization. To finish, in section \ref{recon:sect}, we present some example image reconstructions in two-dimensions and verify our microlocal theory.

\section{Definitions from microlocal analysis}\label{sect:defns} 

We next provide some
notation and definitions.  Let $X$ and $Y$ be open subsets of
{$\mathbb{R}^{n_X}$ and $\mathbb{R}^{n_Y}$, respectively.}  Let $\Dc(X)$ be the space of smooth functions compactly
supported on $X$ with the standard topology and let $\mathcal{D}'(X)$
denote its dual space, the vector space of distributions on $X$.  Let
$\Ec(X)$ be the space of all smooth functions on $X$ with the standard
topology and let $\mathcal{E}'(X)$ denote its dual space, the vector
space of distributions with compact support contained in $X$. Finally,
let $\Sc(\rn)$ be the space of Schwartz functions, that are rapidly
decreasing at $\infty$ along with all derivatives. See \cite{Rudin:FA}
for more information.

For a function $f$ in the Schwartz space $\Sc(\mathbb{R}^{n_X})$ or in
$L^2(\rn)$, we use $\mathcal{F}f$ and $\mathcal{F}^{-1}f$ to denote
the Fourier transform and inverse Fourier transform of $f$,
respectively (see \cite[Definition 7.1.1]{hormanderI}).  Note that
$\Fc\inv \Fc f(\vx)= \frac{1}{(2\pi)^{n_X}}\int_{\vy\in\mathbb{R}^{n_X}}\int_{\vz\in
\mathbb{R}^{n_X}} \exp((\vx-\vz)\cdot \vy)\,
f(\vz)\d \vz\d \vy$.

We use the standard multi-index notation: if
$\alpha=(\alpha_1,\alpha_2,\dots,\alpha_n)\in \sparen{0,1,2,\dots}^{n_X}$
is a multi-index and $f$ is a function on $\mathbb{R}^{n_X}$, then
\[\partial^\alpha f=\paren{\frac{\partial}{\partial
x_1}}^{\alpha_1}\paren{\frac{\partial}{\partial
x_2}}^{\alpha_2}\cdots\paren{\frac{\partial}{\partial x_{n_X}}}^{\alpha_{n_X}}
f.\] If $f$ is a function of $(\vy,\vx,\vs)$ then $\partial^\alpha_\vy
f$ and $\partial^\alpha_\vs f$ are defined similarly.

  We identify cotangent
spaces on Euclidean spaces with the underlying Euclidean spaces, so we
identify $T^*(X)$ with $X\times \mathbb{R}^{n_X}$. If $\Phi$ is a function of $(\vy,\vx,\vs)\in Y\times X\times \rr^N$
then we define $\dd_{\vy} \Phi = \paren{\partyf{1}{\Phi},
\partyf{2}{\Phi}, \cdots, \partyf{{n_X}}{\Phi} }$, and $\dd_\vx\Phi$ and $
\dd_\vs \Phi $ are defined similarly. Identifying the cotangent space with the Euclidean space as mentioned above, we let $\dd\Phi =
\paren{\dd_{\vy} \Phi, \dd_{\vx} \Phi,\dd_\vs \Phi}$.


We use the convenient notation that if $A\subset \rr^m$, then $\dot{A}
= A\smo$.

The singularities of a function and the directions in which they occur
are described by the wavefront set \cite[page
16]{duistermaat1996fourier}: 
\begin{definition}
\label{WF} Let $X$ be an open subset of $\rn$ and let $f$ be a
distribution in $\mathcal{D}'(X)$.  Let $(\vx_0,\vxi_0)\in X\times
\drn$.  Then $f$ is \emph{smooth at $\vx_0$ in direction $\vxio$} if
there exists a neighborhood $U$ of $\vx_0$ and $V$ of $\vxi_0$ such
that for every $\Phi\in \Dc(U)$ and $N\in\mathbb{R}$ there exists a
constant $C_N$ such that for all $\vxi\in V$,
\begin{equation}
\left|\Fc(\Phi f)(\lambda\vxi)\right|\leq C_N(1+\abs{\lambda})^{-N}.
\end{equation}
The pair $(\vx_0,\vxio)$ is in the \emph{wavefront set,} $\wf(f)$, if
$f$ is not smooth at $\vx_0$ in direction $\vxio$.
\end{definition}
 This definition follows the intuitive idea that the elements of
$\WF(f)$ are the point--normal vector pairs above points of $X$ at
which $f$ has singularities.  For example, if $f$ is the
characteristic function of the unit ball in $\mathbb{R}^3$, then its
wavefront set is $\WF(f)=\{(\vx,t\vx): \vx\in S^{2}, t\neq 0\}$, the
set of points on a sphere paired with the corresponding normal vectors
to the sphere.



The wavefront set of a distribution on $X$ is normally defined as a
subset the cotangent bundle $T^*(X)$ so it is invariant under
diffeomorphisms, but we do not need this invariance, so we will
continue to identify $T^*(X) = X \times \rn$ and consider $\WF(f)$ as
a subset of $X\times \drn$.


 \begin{definition}[{\cite[Definition 7.8.1]{hormanderI}}] \label{ellip}We define
 $S^m(Y \times X, \mathbb{R}^N)$ to be the
set of $a\in \Ec(Y\times X\times \mathbb{R}^N)$ such that for every
compact set $K\subset Y\times X$ and all multi--indices $\alpha,
\beta, \gamma$ the bound
\[
\left|\partial^{\gamma}_{\vy}\partial^{\beta}_{\vx}\partial^{\alpha}_{\vsig}a(\vy,\vx,\vsig)\right|\leq
C_{K,\alpha,\beta,\gamma}(1+\norm{\vsig})^{m-|\alpha|},\ \ \ (\vy,\vx)\in K,\
\vsig\in\mathbb{R}^N,
\]
holds for some constant $C_{K,\alpha,\beta,\gamma}>0$. 

 The elements of $S^m$ are called \emph{symbols} of order $m$.  Note
that these symbols are sometimes denoted $S^m_{1,0}$.  The symbol
$a\in S^m(Y \times X,\rr^N)$ is \emph{elliptic} if for each compact set
$K\subset Y\times X$, there is a $C_K>0$ and $M>0$ such that
\bel{def:elliptic} \abs{a(\vy,\vx,\vsig)}\geq C_K(1+\norm{\vsig})^m,\
\ \ (\vy,\vx)\in K,\ \norm{\vsig}\geq M.
\ee 
\end{definition}

\begin{definition}[{\cite[Definition
        21.2.15]{hormanderIII}}] \label{phasedef}
A function $\Phi=\Phi(\vy,\vx,\vsig)\in
\Ec(Y\times X\times\dot{\mathbb{R}^N})$ is a \emph{phase
function} if $\Phi(\vy,\vx,\lambda\vsig)=\lambda\Phi(\vy,\vx,\vsig)$, $\forall
\lambda>0$ and $\mathrm{d}\Phi$ is nowhere zero. The
\emph{critical set of $\Phi$} is
\[\Sigma_\Phi=\{(\vy,\vx,\vsig)\in Y\times X\times\dot{\mathbb{R}^N}
: \dd_{\vsig}\Phi=0\}.\] 
 A phase function is
\emph{clean} if the critical set $\Sigma_\Phi = \{ (\vy,\vx,\vsig) \ : \
\mathrm{d}_\vsig \Phi(\vy,\vx,\vsig) = 0 \}$ is a smooth manifold {with tangent space defined {by} the kernel of $\mathrm{d}\,(\mathrm{d}_\sigma\Phi)$ on $\Sigma_\Phi$. Here, the derivative $\mathrm{d}$ is applied component-wise to the vector-valued function $\mathrm{d}_\sigma\Phi$. So, $\mathrm{d}\,(\mathrm{d}_\sigma\Phi)$ is treated as a Jacobian matrix of dimensions $N\times (2n+N)$.}
\end{definition}
\noindent By the {Constant Rank Theorem} the requirement for a phase
function to be clean is satisfied if
$\mathrm{d}\paren{\mathrm{d}_\vsig
\Phi}$ has constant rank.

\begin{definition}[{\cite[Definition 21.2.15]{hormanderIII} and
      \cite[section 25.2]{hormander}}]\label{def:canon} Let $X$ and
$Y$ be open subsets of $\rn$. Let $\Phi\in \Ec\paren{Y \times X \times
{\rr}^N}$ be a clean phase function.  In addition, we assume that
$\Phi$ is \emph{nondegenerate} in the following sense:
\[\text{$\dd_{\vy}\Phi$ and $\dd_{\vx}\Phi$ are never zero on
$\Sigma_{\Phi}$.}\]
  The
\emph{canonical relation parametrized by $\Phi$} is defined as
\begin{equation}\label{def:Cgenl} \begin{aligned} \Cc=&\sparen{
\paren{\paren{\vy,\dd_{\vy}\Phi(\vy,\vx,\vsig)};\paren{\vx,-\dd_{\vx}\Phi(\vy,\vx,\vsig)}}:(\vy,\vx,\vsig)\in
\Sigma_{\Phi}},
\end{aligned}
\end{equation}
\end{definition}

\begin{definition}\label{FIOdef}
Let $X$ and $Y$ be open subsets of {$\rn$ and $\mathbb{R}^{n_Y}$, respectively.} {Let an operator $A :
\Dc(X)\to \mathcal{D}'(Y)$ be defined by the distribution kernel
$K_A\in \mathcal{D}'(Y\times X)$, in the sense that
$Af(\vy)=\int_{X}K_A(\vy,\vx)f(\vx)\mathrm{d}\vx$. Then we call $K_A$
the \emph{Schwartz kernel} of $A$}. A \emph{Fourier
integral operator (FIO)} of order $m + N/2 - (n_X+n_Y)/2$ is an operator
$A:\Dc(X)\to \mathcal{D}'(Y)$ with Schwartz kernel given by an
oscillatory integral of the form
\begin{equation} \label{oscint}
K_A(\vy,\vx)=\int_{\mathbb{R}^N}
e^{i\Phi(\vy,\vx,\vsig)}a(\vy,\vx,\vsig) \mathrm{d}\vsig,
\end{equation}
where $\Phi$ is a clean nondegenerate phase function and $a$ is a
symbol in $S^m(Y \times X , \mathbb{R}^N)$. The \emph{canonical
relation of $A$} is the canonical relation of $\Phi$ defined in
\eqref{def:Cgenl}. The FIO $A$ is \emph{elliptic} if its symbol is elliptic.
\end{definition}

This is a simplified version of the definition of FIOs in \cite[section
2.4]{duist} or \cite[section 25.2]{hormander} that is suitable when there are global coordinates and a global phase function. In general, an FIO must be defined using a partition of unity, local coordinates, and phase functions corresponding to local regions of the same, globally defined, canonical relation; for details see \cite[section
2.4]{duist} or \cite[section 25.2]{hormander}. Because we assume phase
functions are nondegenerate, our FIOs can be defined as maps from
$\Ec'(X)$ to $\Dc'(Y)$ and sometimes on larger domains. For general
information about FIOs, see \cite{duist, hormander, hormanderIII}. For
information about the Schwartz Kernel, see \cite[Theorem
5.1.9]{hormanderI}. 

Pseudodifferential operators are a special class of FIOs, which include linear differential operators, given in the next definition.

\begin{definition}\label{PsiDOdef}
An FIO is a pseudodifferential operator if its canonical relation $\Cc$ is contained in the diagonal
\[
\Cc=\Delta := \{ (\vx,\vxi;\vx,\vxi)\}.
\]
\end{definition}

{Let $X$ and $Y$ be
sets and let $\Omega_1\subset X$ and $\Omega_2\subset Y\times X$. The composition $\Omega_2\circ \Omega_1$ and transpose $\Omega_2^t$ of $\Omega_2$ are defined
\[\begin{aligned}\Omega_2\circ \Omega_1 &= \sparen{\vy\in Y\st \exists \vx\in \Omega_1,\
(\vy,\vx)\in \Omega_2}\\
\Omega_2^t &= \sparen{(\vx,\vy)\st (\vy,\vx)\in \Omega_2}.\end{aligned}\]}
The H\"ormander-Sato Lemma  provides the relationship between the
wavefront set of distributions and their images under FIO.

\begin{theorem}[{\cite[Theorem 8.2.13]{hormanderI}}]\label{thm:HS} Let $f\in \Ec'(X)$ and
let ${A}:\Ec'(X)\to \Dc'(Y)$ be an FIO with canonical relation $\Cc$.
Then, $\wf({A}f)\subset \Cc\circ \wf(f)$.\end{theorem}

Let $A$ be an FIO with adjoint $A^*$. Then if $\Cc$ is the canonical relation of $A$, the canonical relation of $A^*$ is $\Cc^t$. Many imaging techniques are based on application of the adjoint operator $A^*$ and so to understand artifacts we consider $A^* A$ (or, if
$A$ does not map to $\Ec'(Y)$, then $A^* \psi A$ for an
appropriate cutoff $\psi$). Because of Theorem \ref{thm:HS},
\[
\wf(A^* \psi A f) \subset \Cc^t \circ \Cc \circ \wf(f).
\]
The next two definitions provide tools, which we will apply in the next section, to analyze this composition.

\begin{definition}
\label{defproj} Let $\Cc\subset T^*(Y\times X)$ be the canonical
relation associated to the FIO ${A}:\mathcal{E}'(X)\to
\mathcal{D}'(Y)$. We let $\Pi_L$ and $\Pi_R$ denote the natural left-
and right-projections of $\Cc$, projecting onto the appropriate
coordinates: $\Pi_L:\Cc\to T^*(Y)$ and $\Pi_R : \Cc\to T^*(X)$.
\end{definition}

Because $\Phi$ is nondegenerate, the projections do not map to the
zero section.  
%
%
If $A$ satisfies our next definition, then $A^* A$ (or $A^* \psi A$) is a pseudodifferential operator
\cite{GS1977, quinto}.

\begin{definition}\label{def:bolker} Let
${A}:\Ec'(X)\to \Dc'(Y)$ be a FIO with canonical relation $\Cc$ then
{$A$} (or $\Cc$) satisfies the \emph{Bolker Condition} if
the natural projection $\Pi_L:\Cc\to T^*(Y)$ is an embedding
(injective immersion).\end{definition}

\section{Ellipsoid and hyperboloid Radon transforms}\label{sect:theory} In
this section we show under fairly weak assumptions that a general
Radon transform integrating over ellipsoids, hyperboloids, or
elliptic hyperboloids with centers on a surface satisfies the Bolker
condition. Then, we investigate several special cases.



Let $\sym$ denote the set of invertible symmetric matrices with real
entries, which is an $n(n+1)/2$ dimensional smooth manifold, and
suppose $A \in \sym$. Let $S$ be a smooth connected hypersurface in
$\rn$. For $(\vs,A,t)\in S\times \sym\times \mathbb{R}=:Y$, let 
\begin{equation}\label{eq:PsiA-v2}
\Psi(\vs, A, t;\vx)= t - \vx_T^T A \vx_T \ \ \text{where}\ \ \vx_T =
\vx - \vs.
\end{equation} 
If $A$ is positive definite and $t>0$, then 
 $\Psi(\vs,A,t;\vx)=0$ is the defining equation of an
ellipsoid with center at  $\vs$. In other cases, $\Psi(\vs,A,t;\vx)=0$
can be a hyperboloid or elliptic hyperboloid. Note that if
$t=0$, the surface
$\Psi(\vs,A,0;\vx)=0$ is singular.  Therefore, we will exclude $t=0$
from our analysis.


Our Radon transform can be written
\begin{equation} \label{eq:RdefA}
\begin{split}
R f(\vs,A,t)&=\int_{\mathbb{R}^n}\abs{\nabla_{\vx}\Psi}\delta\paren{\Psi(\vs,
A,t;\vx)}f(\vx)\mathrm{d}\vx\\
&=\int_{-\infty}^{\infty}\int_{\mathbb{R}^n}
\abs{\nabla_{\vx}\Psi}f(\vx)e^{i\sigma
\Psi(\vs,A,t;\vx)}\mathrm{d}\vx\mathrm{d}\sigma
\end{split}
\end{equation}
for $f\in L^2_c(D)$, where $D$ is an open, connected subset of $\rn$. The most general case we will consider is when
$A$ is restricted to be in an embedded submanifold $M \subset \sym$.
Note that this includes the case when $M=\{A\}$ is a single matrix and
thus a zero dimensional submanifold. With this in mind, we define
\[Y_M = S \times M \times \dR\] and the operator $R_M$ is given
by \eqref{eq:RdefA} but with $A$ restricted to $M$. 

We now state our main theorem.

\begin{theorem}\label{thm:RA-general}  Let $S \subset \rn$ be a smooth
connected hypersurface. Let $D$ be an open connected subset of $\rn$,
and let $M$ be a submanifold of $\sym$, possibly of dimension zero.
Then, $R_M:\Ec'(D)\to \Dc'(Y_M)$ is an FIO satisfying the Bolker
condition if $D$ is disjoint from every tangent plane to
$S$.  That is, \bel{no tangents} D\bigcap \paren{\bigcup_{\vs\in S}
P_\vs} = \emptyset, \ee where $P_\vs$ is the tangent plane to $S$ at
$\vs\in S$. If additionally $\mathrm{dim}(M) = 0$, then the Bolker condition will fail if any tangent plane to $S$ intersects $D$.
\end{theorem}

We should point out that Theorem \ref{thm:RA-general} will apply to
the Radon transform $R_M$ with any smooth weight, not just the weight
in \eqref{eq:RdefA}, since the proof uses only microlocal results and
the symbol of $R_M$ will still be smooth.

In the proof of Theorem \ref{thm:RA-general} and throughout the
article, we use the following notation: $\zero_{m\times n}$ is the
$m\times n$ zero matrix; and $I_m$ is the $m\times m $ identity
matrix. If $\vx= (x_1,x_2,\dots, x_{n-1},x_n)\in \rn$, then $\vx' =
(x_1,x_2,\dots, x_{n-1})$. 

\begin{proof}[Proof of Theorem \ref{thm:RA-general}]
Referring to the second line in \eqref{eq:RdefA}, $R_M$ will be an FIO provided that
\begin{equation}
\label{phiA}
\Phi(\vs,A,t;\vx;\sigma)=\sigma\paren{ t - \vx_T^T A\vx_T}
\end{equation}
is a nondegenerate phase function. This is true since $\frac{\partial
\Phi}{\partial t} = \sigma$ and $\nabla_\vx \Phi = -2\sigma
(A\vx_T)^T\neq \zero$ since $D$ is disjoint from $S$ and $A$ is
invertible.

Our proof is in two parts. First we consider the case when $S$ is the
graph of a smooth function, and then we use this result locally for
the general case. 

Indeed, let $\Omega$ be an open connected
subset of $\rnmo$ and let $q:\Omega\to\rr$ be a smooth function. Now,
let 
\[S =\sparen{(\vy,q(\vy))\st\vy\in \Omega}.\] 
To simplify notation when $\vy\in \Omega$ is fixed, we will let
\[q=q\paren{y_1,\ldots,y_{n-1}},\qquad q_j = \frac{\partial q}{\partial
y_j}\] and so for $\vx\in \mathbb{R}^n$, $\vy \in \Omega$
\[\vx_T = (x_1-y_1,x_2-y_2,\dots,x_{n-1}-y_{n-1},x_n-q)^T.\]
Note that, with this notation, $\vx$ is in the tangent plane
$P_{(\vy,q(\vy))}$ if and only if
\begin{equation}\label{eq:tangentplane}
(\nabla q^T, -1) \cdot \vx_T = (q_1,\ q_2,\ \dots \ ,\ q_{n-1},\ -1) \cdot \vx_T = 0.
\end{equation}


When $\mathrm{dim}(M)>0$, calculation using \eqref{phiA} shows that the canonical relation for $R_M$ is 
\begin{equation}\label{def:CM}
\begin{aligned}
\Cc_M &=
\big\{\paren{(\vy,q),A, \vx_T^T A \vx_T; \nabla_\vy
\Phi \cdot \dd \vy + \nabla_A \Phi \cdot \dd A+ \sigma \dd t; \vx;- \nabla_\vx
\Phi\cdot \dd \vx}\\
&\hskip6cm \st (\vy,A;\vx;\sigma)\in \Omega\times M\times D\times
\dR\big\}
\end{aligned}
\end{equation}
The only difference when $\mathrm{dim}(M)=0$ is that the $\dd A$ term
is removed. For the remainder of the proof we assume
$\mathrm{dim}(M)>0$ but minor modifications allow the same
arguments to work for the case $\mathrm{dim}(M)=0$.

Note that $(\vy,A;\vx;\sigma)\in \Omega\times M \times D\times \dR$
provide a global parametrization of $\Cc_M$ since $t$ is determined by
$\vy$, $q(\vy)$, $\vx$ and $A$. The left projection of $R_M$ is
\begin{equation}\label{PiA}
\begin{split}
\Pi^{M}_L(\vy,A;\vx;\sigma)&=\Bigg(
(\vy,q),A,
\vx_T^T A \vx;  2\sigma \vx_T^TA B^T \mathrm{d}\vy 
-\sigma\vx_T^T \mathrm{d}A \vx_T + \sigma \mathrm{d}t\Bigg),
\end{split}
\end{equation}
where 
\[B = \bparen{I_{n-1},\nabla
q}=\begin{pmatrix}1&0&\cdots&0&q_1\\
0&1&\cdots&0&q_2\\
\vdots&\vdots&\vdots&\vdots&\vdots\\
0&0&\cdots&1&q_{n-1}\end{pmatrix}.
\]
Using the natural coordinates on $T^* (\Omega \times M \times \dR)$, the differential of $\Pi^M_L$ is represented by
\begin{equation}\label{DPiA}
D\Pi^{M}_L=\kbordermatrix {&\nabla_{\vy},\nabla_{A},
\frac{\partial}{\partial \sigma}& \nabla_\vx \\
\vy,A, \dd t& I_{(2n-1)} & \textbf{0}_{(2n-1) \times n} \\
\dd \vy & \cdot & 2\sigma BA\\
\dd A & \cdot & \cdot \\ 
{t} & \cdot & 2\vx_T^TA}.
\end{equation}
Thus, $\Pi^M_L$ will be an immersion if the $n\times n$ submatrix
\begin{equation}\label{nablaxA}
C= 
\begin{pmatrix}2\sigma B A\\
 2\vx_T^TA \end{pmatrix}
 =
 \begin{pmatrix}2\sigma B\\
 2\vx_T^T \end{pmatrix} A
\end{equation} 
is invertible. Note that $CA^{-1}$ is row equivalent to 
\begin{equation}\label{reducedA}
(CA^{-1})'= 
\begin{pmatrix}
I_{n-1}&\nabla q\\
\zero_{1\times (n-1)}&(\nabla q^T,  -1)\cdot \vx_T
  \end{pmatrix}.  
\end{equation}
Therefore, by \eqref{eq:tangentplane} $D\Pi^M_L$ is injective if $\vx$
is not in the tangent plane to $S$ at $(\vy,q)$. If $\mathrm{dim}(M) = 0$, then the rows of \eqref{DPiA} corresponding to $\dd A$ are absent and so $D\Pi^M_L$ is injective if and only if $\vx$ is not in the tangent plane to $S$ at $(\vy,q)$. This proves the final statement of the theorem about the Bolker condition failing. On the other hand, if no tangent plane to $S$ intersects $D$ then this shows that $D \Pi^M_L$ is always injective which proves that $\Pi^M_L$
is an immersion in that case.

Now we will prove the injectivity part of the Bolker condition when
$S$ is a graph assuming that \eqref{no tangents} holds. Without loss
of generality, we assume $D$ is above every tangent plane to $S$,
i.e., if $\vs\in S$, $\vx\in D$, and $(\vx',t)\in P_\vs$, then $x_n
>t$. To see this is possible, one assumes some tangent plane
$P_{\vs_1}$ is below $D$ and another, $P_{\vs_2}$ is above. Then, one
uses the following  Intermediate Value Theorem argument to show some tangent plane
intersects $D$. Let $\vx\in D$ and let $\ell = \sparen{(\vx',w)\st
w\in \rr}$ be the vertical line containing $\vx$. The point of
intersection $P_{\vs_1}\cap \ell$ is below $D\cap\ell$ and
$P_{\vs_2}\cap \ell$ is above. Therefore, there is an $\vs\in S$ where
$P_\vs\cap \ell$ is in $D$ since both $S$ and $D$ are connected and
the map from $S$ to the point of intersection of $\ell$ and $ P_\vs$
is continuous. This assumption implies that 
\begin{equation}\label{eq:locconv}
(-q_1, -q_2, \dots,-q_{n-1}, 1) \cdot \vx_T >0
\end{equation}
whenever $\vx \in D$ and $\vy \in \Omega$. 

Seeking to establish injectivity, let us suppose that $\vu,\vv\in D$, $A \in M$ and $\sigma \in \dR$ are such that
\bel{equal PiL}
\Pi^{M}_L(\vy,A;\vu;\sigma)=\Pi^{M}_L(\vy,A;\vv;\sigma).\ee Then,
using \eqref{PiA}, we see
\begin{equation}\label{eq:uvEqA}
BA \vu_T=B A \vv_T.
\end{equation}
Note that $\text{Null}\paren{BA}=\text{span}\paren{A^{-1}(-\nabla
q^T,1)^T}$ and because of \eqref{eq:uvEqA}
\begin{equation}\label{eq:uvEq2A}
\vu_T = \vv_T + sA^{-1} (-\nabla q^T,1)^T.
\end{equation}
for some $s \in \mathbb{R}$. On the other hand, by setting the $t$
components in \eqref{PiA} equal we have
\begin{equation}\label{eq:uvEq3A}
\vu^T_T A \vu_T = \vv^T_T A \vv_T
\end{equation}
By taking the inner product of \eqref{eq:uvEq2A}  with $A\vu_T$
and using \eqref{eq:uvEq3A}, we see that 
\bel{eq:uvEq4A} -s\vv_T\cdot (-\nabla q^T,1)^T =s \vu_T\cdot(-\nabla
q^T,1)^T.\ee Therefore, either $s=0$ and $\vu = \vv$, or $\vv$ and $\vu$ are
on opposite sides of the tangent plane $P_{(\vy,q)} $ to $S$ at $(\vy,q)$. 
This second case is not allowed since $D$ is above every tangent plane
to $S$.

Now we consider the general case when $S$ is a smooth connected
hypersurface, not necessarily a graph, and $D$ an open connected set.

For the last statement of the theorem concerning when the Bolker
condition fails, if there is a point in $D$ which is in a tangent
plane $P_{\vs_0}$ to $S$, then after translating and rotating $S$ will
be represented by a graph in a neighborhood of $\vs_0$ and the
 argument after \eqref{reducedA} shows the immersion part of the Bolker condition
will fail if $\mathrm{dim}(M) = 0$.

With this case handled, we now assume that no tangent plane to $S$ intersects $D$ and suppose $\vso\in S$. We will show that $\Pi_L^M$ is an injective immersion locally above $S$
near $\vso$ (i.e., on the canonical relation $\Cc_M$ and above points
$(\vs,A,t)\in Y_M$ for $\vs$ in a neighborhood in $S$ of $\vso$). We
do this by reducing the problem to the case we just considered, when
the hypersurface $S$ is a graph.

This will show $R_M:\Ec'(D)\to
\Dc'(Y_M)$ satisfies the Bolker assumption globally for the following
reasons. First, being an immersion is a local condition. To check
injectivity, note that if
$\Pi_L^M(\nu_0)=(\vso,A_0,t_0;\eta_0)=\Pi_L^M(\nu_1)$ then, since the
basepoints of the image are the same, to show $\nu_0=\nu_1$, one just
needs to know $\Pi_L^M$ is injective on
$\paren{\Pi_L^M}\inv\sparen{(\vso,A_0,t_0;\eta_0)}$.
\medskip

Using a translation $T$ of $\rn$ followed by a rotation $R$ we map
$\vso$ into $\zero\in \rn$ and $S$ into a connected submanifold $S'$
such that the hyperplane $P_\zero=\sparen{\vx\in \rn\st x_n = 0}$ is
tangent to $S'$ at $\vs = \zero$. We let $D'$ be the image of $D$
under this rigid motion, $RT$. Without loss of generality, we assume
that $D'$ is above $P_\zero$. The rotation $R$ also conjugates $M$ to another
embedded manifold in $\sym$, which we denote by $M'$.

 Let $\Omega$ be an open connected neighborhood of $\zero\in \rnmo$ that
is so small that there is a smooth function $q:\Omega\to [0,\infty)$
such that \[\Omega\ni \vy \mapsto (\vy,q(\vy))\] give local
coordinates on $S'$ near $\zero$. Let \[S'_0 = \sparen{(\vy,q(\vy))\st
\vy\in \Omega}.\]  Since $D'$ is above $P_\zero$, it is above every
tangent plane to $S'_0$ by the Intermediate Value Theorem
argument given earlier.


Let $Y'_0 = \sparen{(\vy,q(\vy)),A,t)\st \vy\in \Omega, A\in M', t\in
\dR}$. Since $D'$ is above every tangent plane to $S'_0$, the first
part of this proof implies that $R_{M'}:\Ec'(D')\to \Dc'(Y'_0)$
satisfies the Bolker condition. Let $S_0$ be the image of $S'_0$ under
$(RT)\inv$ and let $Y_0 = S_0\times M\times \dR$. This proof implies
that $R_M:\Ec'(D)\to \Dc'(Y_0)$ satisfies the Bolker condition. Since
the Bolker condition is local above $Y_M$ and these coordinate patches
cover $S$, the Bolker condition holds for $R_M:\Ec'(D)\to \Dc'(Y_M)$.
This finishes the proof.\end{proof}

\subsection{Visible singularities and artifacts}\label{sect:visible} 

 The normal operator for $R_M$ is $\Nc = R_M^*\varphi R_M $ where
$\varphi$ is a cutoff on $Y_M$ that guarantees one can compose these
operators (sometimes $R_M^*$ is replaced by a weighted dual operator,
and if $R_M$ is a proper map, the cutoff $\varphi$ is not needed). 

\emph{Visible singularities of $f$} are those that are singularities
of $\Nc f$, i.e., singularities in $\wf(\Nc f)\cap \wf(f)$. Other
singularities of $f$ are called \emph{invisible singularities}.
\emph{Artifacts} are singularities of $\Nc f$ that are not in $f$,
i.e., singularities in $\wf(\Nc f)\setminus \wf(f)$.

To understand visible singularities, note that the set of visible
singularities of $f$ is contained in the image of $\Pi^M_R(\Cc_M)$.
This is true because \bel{HS-Nf}\wf(\Nc f)\subset \Cc_M^t\circ
\Cc_M\circ\wf(f)=\Pi_R^M\circ \paren{\Pi_L^M}\inv \circ \Pi_L^M\circ
\paren{\Pi_R^M}\inv(\wf(f)),\ee by the H\"ormander-Sato Lemma
\cite[Theorem 8.2.13]{hormanderI} and so the only singularities that
will come from this composition will be those in $\Pi_M^R(\Cc_M)$.

Now, we consider visible singularities for the spherical transform, so
$M=\sparen{I_n}$. We claim, there will be more  visible
singularities the more $S$ ``wraps around'' the scanning region. To
see this, one first observes that a singularity $(\vx,\xi)\in \wf(f)$
is detected by spherical integrals  only if $\{\vx+\alpha \xi \st
\alpha \in\mathbb{R}\}\cap S\neq \emptyset$. This follows by a
calculation of $\Cc_M$ for this case (either use the expression for
$\Cc_M$ in \eqref{def:CM} for the spherical case and then find the
image of $\Pi_R^M(\Cc_M)$ or see e.g., the discussion of visible
singularities around (4.6) and (4.13) in \cite{FrQu2015}). Therefore,
if $S$ is as in Figure \ref{F1b}, then $R_M f$ detects all
singularities in $D$ since $S$ surrounds $D$. If $S$ is as in Figure
\ref{F1c}, then singularities in a horizontal direction at points
above $S$ are invisible since no  horizontal line through such points
intersects $S$



%

%

We can also use the proof of Theorem \ref{thm:RA-general} to understand how
one can get artifacts in reconstructions using $\Nc$ if $S$ is not
globally convex or $D$ is not on the convex side of $S$. For
simplicity, we consider the case when $S$ is the graph of a function,
$q:\Omega\to\rr$ where $\Omega $ is an open subset of $\rnmo$.

The injectivity part of the Bolker condition fails if points below and
above $S$ are in the domain $D$ or if $S$ is not globally convex. Let
$(\vy,A,t)\in Y_M$ and let \[ E=E(\vy,A,t)=\sparen{\vx\st
\vx_T^TA\vx_T=t}.\] Then $E$ is the manifold of integration for $R_M$
at $ (\vy,A,t)$. 

Let $\vx\in E$ and assume $\xi$ is conormal to $E$ at $\vx$. Then
$\tau = (\vx,\xi)\in \Pi_R^M(\Cc_M)$ by \eqref{def:CM} and
$\nu=(\vy,A,t,\eta) \in 
\Cc_M\circ\sparen{\tau}$ for some $\eta\in T^*_{(\vy,A,t)}(Y_M)$.

By the injectivity calculation for $\Pi_L^M$, there is a second preimage
of $\nu$ and its basepoint is calculated using equations
\eqref{eq:uvEq2A}, \eqref{eq:uvEq3A}, and \eqref{eq:uvEq4A}. This
``mirror point'' $\vx_m$ to $\vx$ is the other point on $E$ and the line
through $\vx$ parallel to $A\inv(-\nabla q^T,1)^T$. 
This mirror point, $\vx_m$ is on
the other side of the tangent plane to $S$ at $(\vy,q)$ by
\eqref{eq:uvEq4A}.

Note that $\vx_m$ is the basepoint of a second preimage, $\tau' =
(\vx_m,\xi')$, of $\nu$ under composition with $\Cc_M$, i.e., $\nu\in
\Cc_M\circ \sparen{\tau}$ and $\nu\in
\Cc_M\circ \sparen{\tau'}$ by \eqref{equal PiL}.

If $\vx$ is on the tangent plane to $S$ at $(\vy,q)$, then $\vx_m=\vx$
(see \eqref{eq:uvEq4A}),and $\Pi_L^M$ is not an immersion at
$(\nu,\tau)$, and the immersion assumption of Bolker breaks down here.

This explains why artifacts can occur if $S$ is not globally convex or
if $D$ is on both sides of $S$. Given any point $\vx \in E(\vy,A,t)$,
there is a mirror point $\vx_m$ on the other side of the tangent plane
to $S$ at $(\vy,q)$ such that both covectors introduced in this
section, $\tau$ and $\tau'$ are in $\Cc_M^t\circ\sparen{\nu}$.
Therefore, if $\tau\in \wf(f)$, by \eqref{HS-Nf}, $\tau'$ can be an
added singularity in $\wf(\Nc f)$, and there can be an artifact at
$\vx_m$.

%

\subsection{Examples}\label{sect:examples}
 In this section, we apply Theorem \ref{thm:RA-general} to several
 interesting special cases.

\begin{corollary} \label{cor:RA-convex} Let $C$ be an open convex set with a smooth
boundary, $S$, and let $M$ be a submanifold of $\sym$, possibly of
dimension zero. Then, $R_M:\Ec'(C)\to \Dc'(Y_M)$ is an FIO satisfying
the Bolker condition.\end{corollary}

The corollary follows because condition \eqref{no tangents} in Theorem
\ref{thm:RA-general} holds as $C$ is convex and $S$ is smooth.

\begin{example}[$S$ with gradient zero along an axis]\label{ex:cylinder}
In this example, we consider measurement surfaces $S$ which are flat
along one directional axis, and the special case when the
integral surfaces are ellipsoids of revolution (or spheroids).
Without loss of generality, in this example, $S$ is assumed to be
flat in the $x_{n-1}$ direction. In Ultrasound reflection tomography
(URT), the integration surfaces are spheroids, and the foci of the
spheroids represent the sound wave emitter/receiver positions
\cite{ABKQ2013}. If we were to construct a measurement surface in
URT, which is flat in the  $x_{n-1}$ direction, such that there is an emitter/receiver at every point on $S$ (i.e., we have an $(n-1)$-D surface of emitters), then the URT data can be modeled by $Rf$, where $f$, in URT, denotes the acoustic reflectivity function. Specifically, we set $A=\text{diag}(1,\ldots,1,r_{n-1},1)$ in equation \eqref{eq:PsiA-v2}, with {$r_{n-1}\in(0,1]$}, and constrain $S$ to have gradient zero in the $x_{n-1}$ direction. Then the defining function, $\Psi$ (see section \ref{sect:theory}, equation \eqref{eq:PsiA-v2}),  describes a spheroid surface with foci on $S$. Thus, Theorem \ref{thm:RA-general} has direct applications to measurement surfaces in URT. 
\begin{figure}[!h]
\centering
\begin{subfigure}{0.32\textwidth}
\includegraphics[width=0.9\linewidth, height=4cm, keepaspectratio]{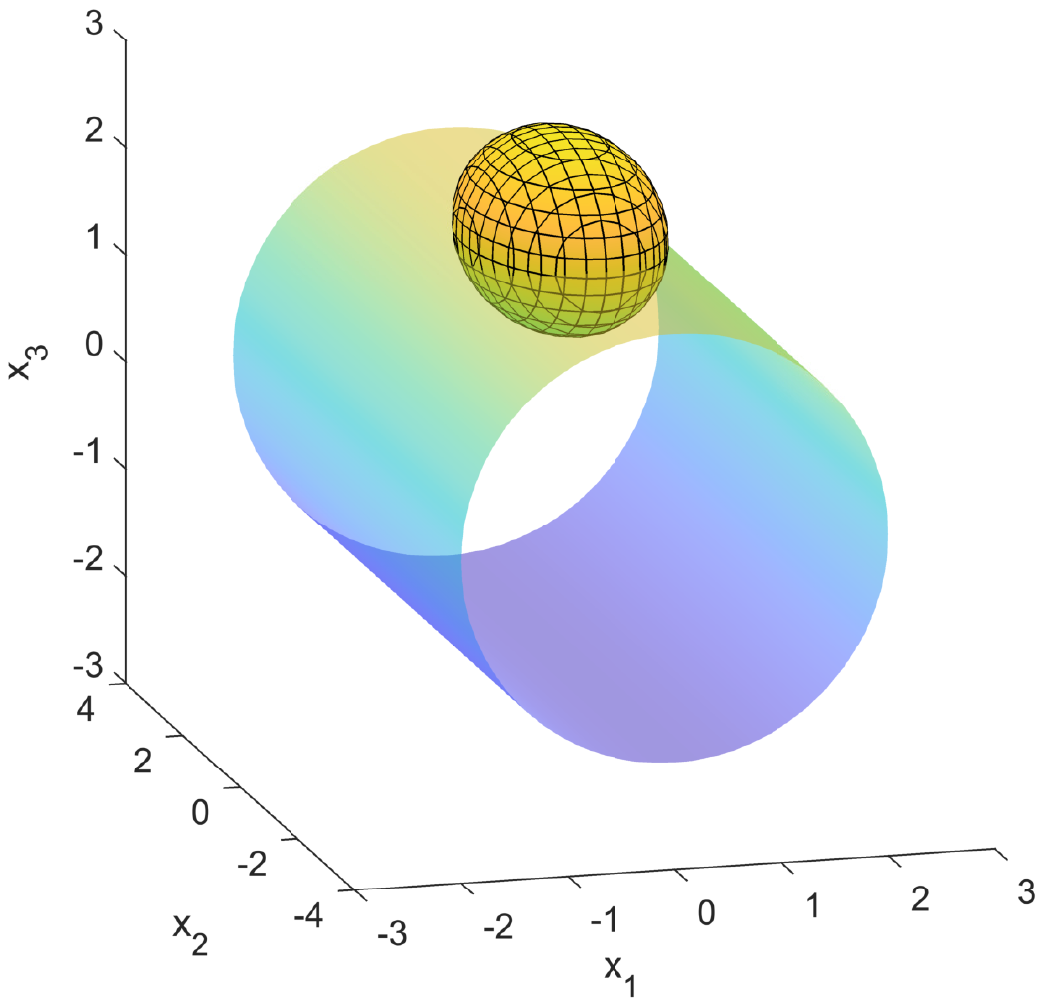}
\subcaption{cylindrical $S$} \label{F2a}
\end{subfigure}
\begin{subfigure}{0.32\textwidth}
\includegraphics[width=0.9\linewidth, height=4cm, keepaspectratio]{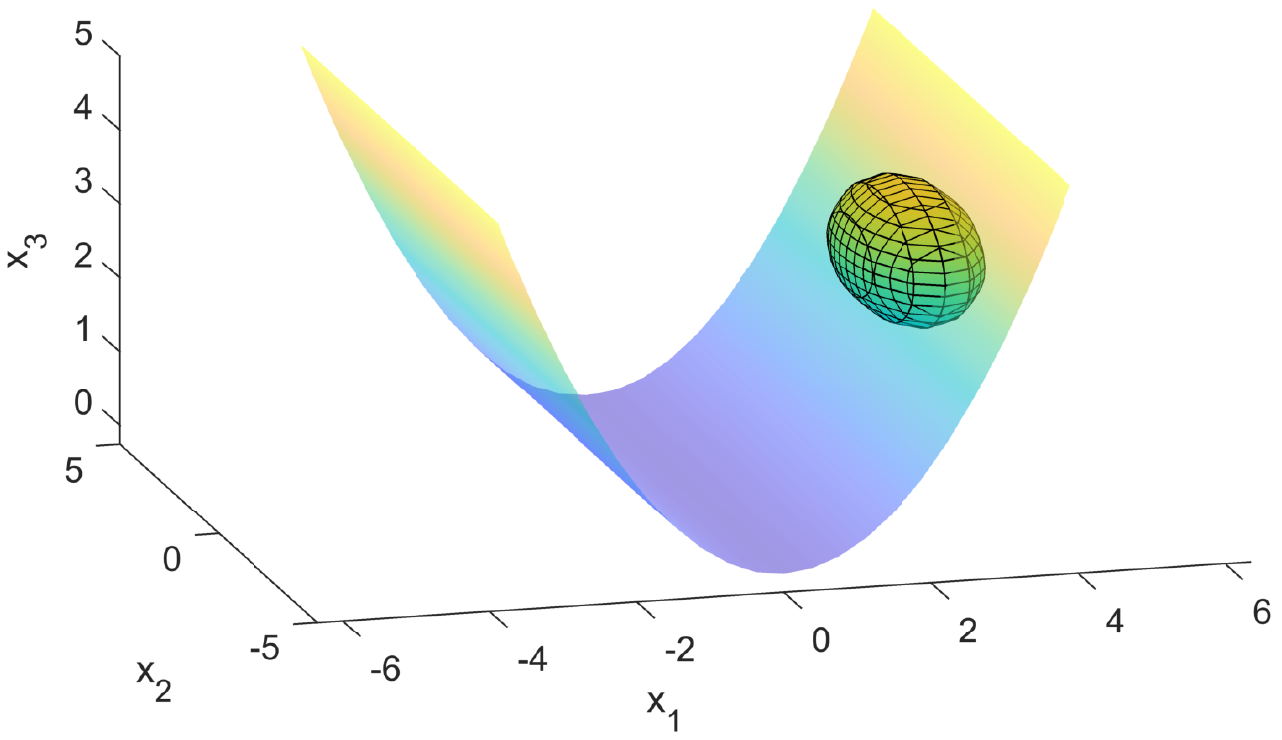} 
\subcaption{parabolic $S$.} \label{F2b}
\end{subfigure}
\caption{Example $S$ in $\mathbb{R}^3$ with practical application to URT. The surfaces above are flat (gradient zero) in the $x_2$ direction and are globally convex.}
\label{F2}
\end{figure}
\end{example}

In Figure \ref{F2}, we have illustrated some example $S \subset \mathbb{R}^3$, which are flat in the $x_2$ direction, and for which the Bolker condition holds. 
In Figure \ref{F2a}, the function support lies in the cylinder interior, and in Figure \ref{F2b}, the function support is assumed to be contained in $\{x_3>x_1^2\}$. 

\begin{example}[Centers on a hyperplane]\label{ex:hyperplane}
Integral transforms over spheres or ellipsoids centered on a plane have been
studied for application to radar \cite{SU:SAR2013,Caday:SAR,krishnan2012microlocal,CokerTewfik}, sonar \cite{andersson1988determination, klein2003inverting}, seismic \cite{grathwohl2020imaging}, and ultrasound imaging \cite{ABKQ2013,gouia2012approximate}.
Theorem \ref{thm:RA-general} holds if $S$ is a hyperplane and $D$ is on one side of
$S$. Note that this does not cover the common offset geometry, where
integrals are being taken over ellipsoids centered on a plane and with
foci a fixed distance apart, or common midpoint, where integrals are
being taken over ellipsoids with foci symmetric about a line cases
considered in \cite{FKNQ:CM-CO-seismics} where integrals are being
taken over ellipsoids with foci a fixed distance apart and oriented 
for each $\vy\in S$, since the foci of
our ellipses change with $t$. 


\begin{figure}[!h]
\centering
\begin{subfigure}{0.32\textwidth}
\includegraphics[width=0.9\linewidth, height=4cm, keepaspectratio]{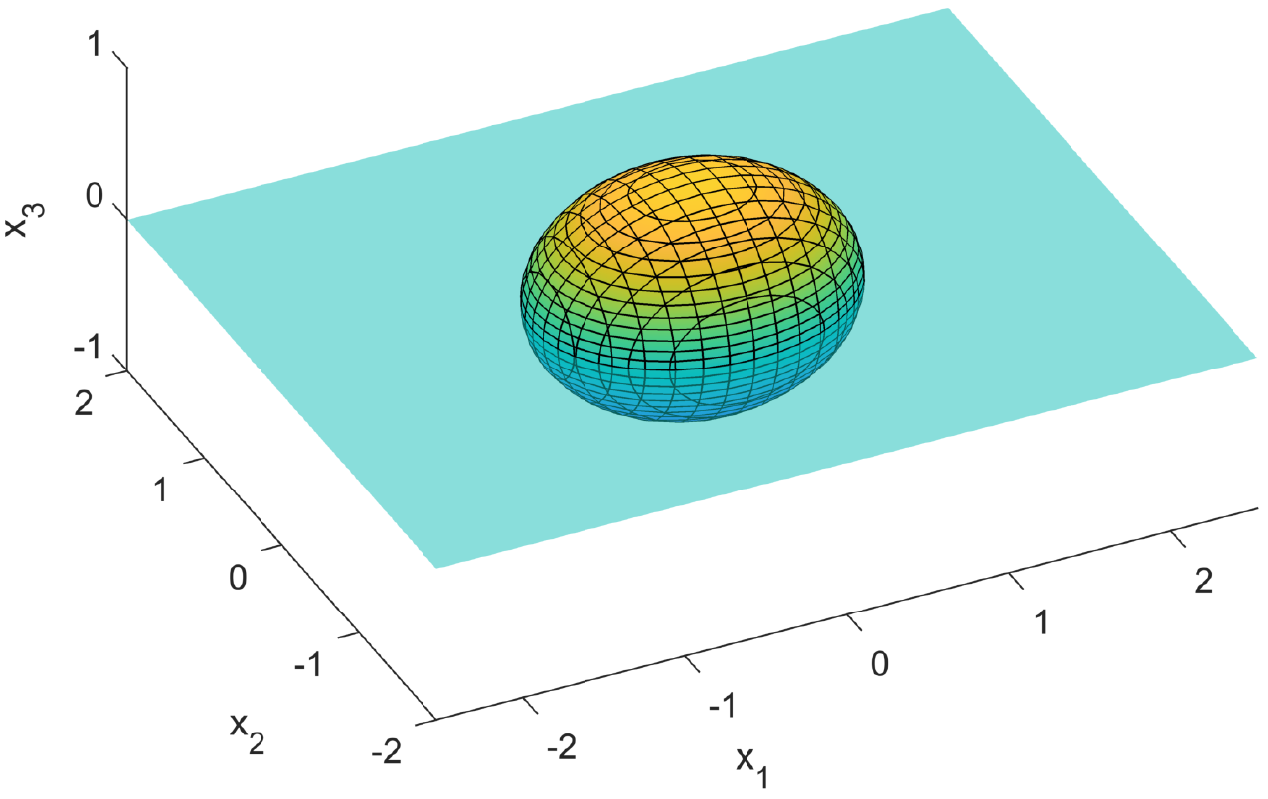}
\subcaption{ellipsoid centered on flat plane} \label{F3a}
\end{subfigure}
\hspace{1cm}
\begin{subfigure}{0.32\textwidth}
\includegraphics[width=0.9\linewidth, height=4cm, keepaspectratio]{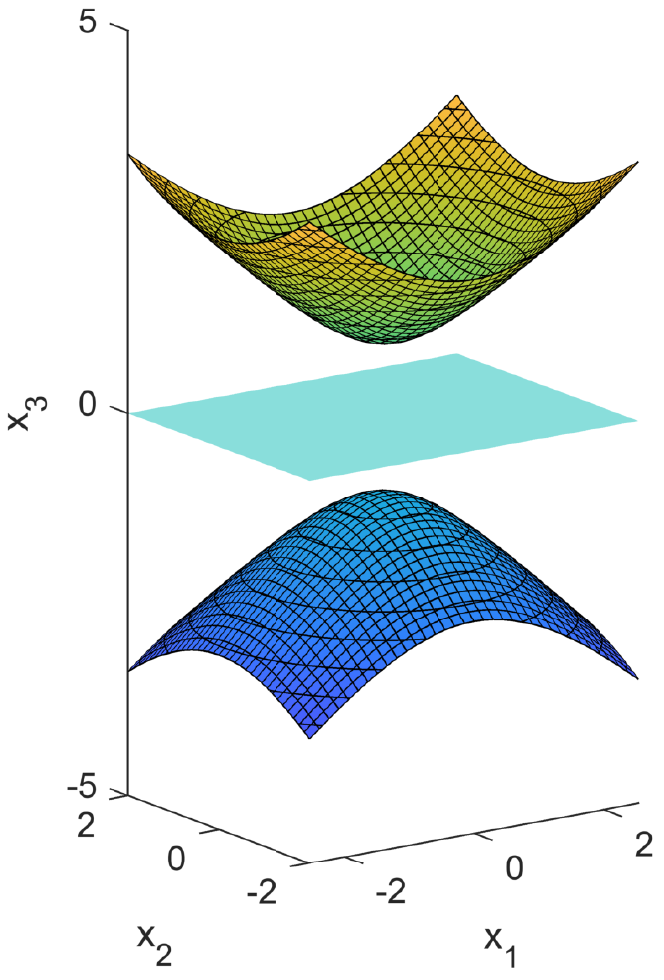} 
\subcaption{two-sheeted hyperboloid centered on flat plane} \label{F3b}
\end{subfigure}
\caption{Flat plane measurement surface. Left - ellipsoid integral surface. Right - two-sheeted hyperboloid surface.} \label{F3}
\end{figure}

\end{example}

\begin{example}[Centers on a spheroid, exponential, and sinusoid surface]\label{ex:other}
In this example, we discuss additional example measurement surfaces
in cases when the Bolker condition is satisfied, and others when
Bolker is not satisfied. Specifically, we consider the spheroid and exponential surfaces
illustrated in Figures \ref{F1a} and \ref{F1b}. In Figure \ref{F1a}, the function support is assumed to be contained within the spheroid interior, and in Figure \ref{F1b}, $S=\{\vx\in\mathbb{R}^3 : x_3-e^{x_1^2+x_2^2}=0\}$, and $\text{supp}(f)\subset \{x_3>e^{x_1^2+x_2^2}\}$. In both cases, no plane tangent to $S$ intersects $\text{supp}(f)$. Therefore, Theorem
\ref{thm:RA-general} holds. 

In Figure \ref{F1c}, we give an example ``sinusoidal" measurement
surface, defined by  $S=\{\vx\in\mathbb{R}^3 : x_3-\sin x_1 +\sin
x_2=0\}$, with $\text{supp}(f) \subset \{\vx\in\mathbb{R}^3 : x_3-\sin x_1 +\sin
x_2 > 0\}$. In this
case, there exist planes tangent to $S$ which intersect $\text{supp}(f)$, and thus, if $\text{dim}(M) = 0$, the Bolker condition is not satisfied by Theorem
\ref{thm:RA-general}, and we would expect to see mirror-point type artifacts through planes tangent to $S$, as described in section \ref{sect:visible}.
\begin{figure}[!h]
\centering
\begin{subfigure}{0.32\textwidth}
\includegraphics[width=0.9\linewidth, height=4cm, keepaspectratio]{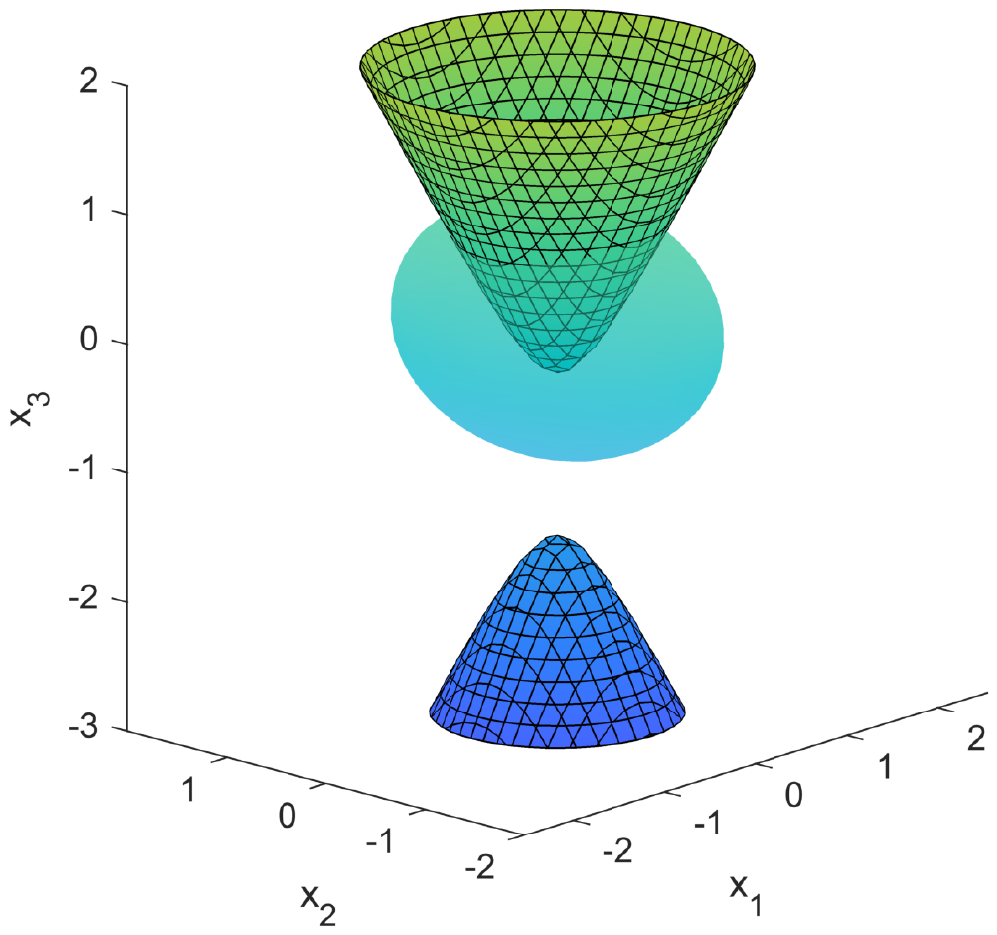}
\subcaption{spheroid} \label{F1a}
\end{subfigure}
\begin{subfigure}{0.32\textwidth}
\includegraphics[width=0.9\linewidth, height=4cm, keepaspectratio]{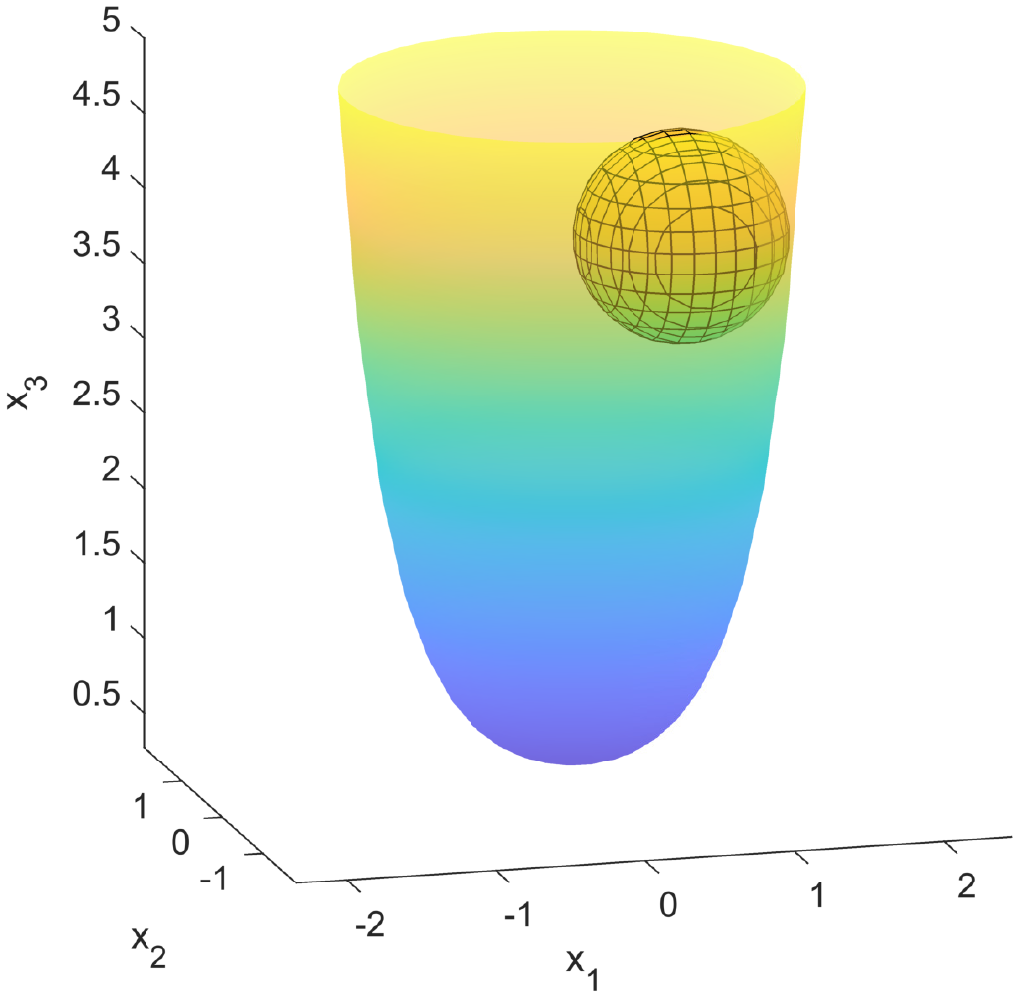} 
\subcaption{exponential} \label{F1b}
\end{subfigure}
\begin{subfigure}{0.32\textwidth}
\includegraphics[width=0.9\linewidth, height=4cm, keepaspectratio]{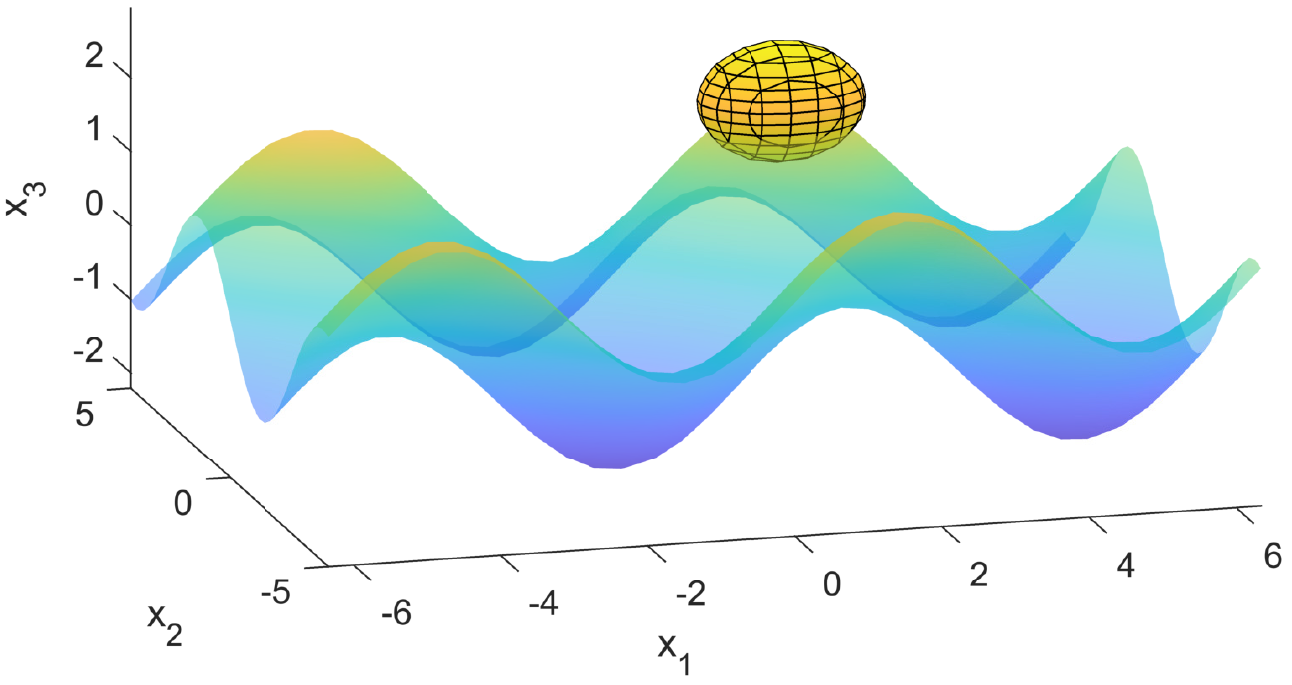}
\subcaption{sinusoid} \label{F1c}
\end{subfigure}
\caption{Two-sheeted hyperboloid and ellipsoid surfaces of
integration centered on convex and non-convex scanning surfaces.
in \eqref{F1c} so Bolker is not satisfied (see section
\ref{sect:visible}). 
} \label{F1}
\end{figure}\end{example}

 \section{Cylindrical measurement surface in $\mathbb{R}^3$}
\label{cylinder:sect}
In this section, we investigate in more detail the cylindrical scanning surface introduced in Example \ref{ex:cylinder} and Figure \ref{F2a}. Specifically, we show that any $L^2$ function, with compact support on a cylinder interior, can be recovered uniquely from its integrals over spheroids with foci on a cylinder. We also discuss in more detail the visible singularities and how the stability varies with the discretization of  emitters/receivers on the cylinder surface.

Our center set will be the cylinder of radius one with axis parallel to the second coordinate axis
and we will consider spheroids with rotation axis on
$C$, centers \bel{def:s}\vs = \vs(\phi_0,y_0) =
(\cos\phi_0,y_0,\sin\phi_0)^T\in C,\ee and fixed aspect ratio $s\in (0,1]$.
This gives the Radon transform 
\begin{equation}\label{def:Rs}
R_sf(p,\phi_0,y_0)=Rf\paren{\vs(\phi_0,y_0),\text{diag}(1,s^2,1),p^2},
\end{equation}
where $t$ is replaced by $p^2=t$ in \eqref{eq:PsiA-v2}.

In the following subsections, we address the injectivity and
microlocal stability properties of $R_s$.

\remark{{Injectivity results are proven in \cite{HomanZhou}, for
a class of generalized Radon transforms for real-analytic submanifolds
in a compact real-analytic manifold with boundary. This important work
does not imply injectivity for $R_s$ for reasons which we now explain.
The key point is that the spheroids $R_s$ integrates over are not
parameterized as in \cite{HomanZhou}. After a reduction, the authors in \cite{HomanZhou}
parameterize their manifolds of integration by $(s,\theta)\in
\rr\times S^{n-1}$ \cite[p.\ 1518]{HomanZhou}, but our spheroids are
parameterized by $(p,\vs)\in \dot{\rr}\times C$. $C$ is
topologically neither compact nor simply connected, although $S^{n-1}$
is. Thus, the theory of \cite{HomanZhou} cannot apply to $R_s$.
To prove injectivity for $R_s$, we apply linear Volterra equation theory. We
also provide an inversion method based on Neumann series. }
}

\subsection{Injectivity}
We first introduce notation we will use in the proofs and define the
auxiliary variables
\begin{equation}
\label{cyl_param} \hxo =\hxo(\phio, x_1,x_3)=
\sqrt{(x_1-\cos\phio)^2+(x_3-\sin\phio)^2}\qquad \hxt =
\hxt(x_2,y_0) =x_2-y_0,\end{equation} then $\hxo$ represents the
radius of the circle in Figure \ref{fig1}. In this notation, we can
describe the spheroid defined by matrix $M = \text{diag}(1,{s^2},1)$,
center $\vs(\phio,y_0)$, and parameter $p$ as
\begin{equation}\label{cyl} \hat{x}_1^2+s^2\hat{x}_2^2=p^2.\end{equation} We use standard cylindrical coordinates for points
inside $C$:
 \[\vx =(x_1,x_2,x_3)=(r\cos\phi,x_2,r\sin\phi), \ \ \text{where}\ \  r>0
\ \ \phi\in[0,2\pi].
\] 
 In
this notation the polar radius for $\vx$ is 
\bel{cyl2}r=\sqrt{\hat{x}_1^2+1-2\hat{x}_1\cos\theta},\ \ \
\hphi = \phi-\phio, \ \ \
\frac{\hat{x}_1}{\sin\hphi}=\frac{r}{\sin\theta},
\ee
where $\hphi$ is the angle between
$(x_1,x_3)$ and $(\cos\phio,\sin\phio)$. See Figure \ref{fig1}. Note that Figure \ref{fig1}
shows the picture for $\phio=0$ and, in general, the picture is
rotated and $\hphi$ is measured from the vector
$(\cos\phio,\sin\phio)$.

We use Figures \ref{fig1} and \ref{fig2} in the proofs to explain the
geometry behind our integrals. They show two cross-sections of the
spheroid \eqref{cyl} with center $\vs( 0,y_0)=(1,y_0,0)$: first with a plane
perpendicular to the $x_2$ axis in Figure \ref{fig1} and second with a plane
containing the axis of the cylinder and $\vs$ in Figure \ref{fig2}.


\begin{figure}[!h]
\centering
\begin{tikzpicture}[scale=0.75]
\draw [->,line width=1pt] (-5,0)--(5,0)node[right] {$x_1$};
\draw [->,line width=1pt] (0,-5)--(0,5)node[right] {$x_3$};
\draw[green,rounded corners=1mm] (0-1,0+1) \irregularcircle{0.75cm}{1mm};
\node at (-0.9-1,1+1) {$f$};
\draw (0,0) circle (4.5);
\draw (4.5,0) circle (2);
\coordinate (x) at (3,1.3229);
\coordinate (O) at (0,0);
\coordinate (c) at (4.5,0);
\draw (O)--(x);
\draw (c)--(x);
\draw pic[draw=orange, <->,"$\ \fnshphi$", angle eccentricity=2] {angle = c--O--x};
\draw pic[draw=orange, <->,"$\theta$", angle eccentricity=1.5] {angle = x--c--O};
\node at (1.5,0.9) {$r$};
\node at (4,1) {$\hat{x}_1$};
\node at (2.25,-0.3) {1};
\end{tikzpicture}
\caption{A cross section of the spheroid with center $(1,0,0)=\vs(0,0)$
perpendicular to the spheroid axis at $\hxt=x_2$.  The axes are
$(x_1,x_3)$ because $\phio=0$ but the picture would be rotated for
general  $\phio$ and then the angle $\hphi$ would be measured  from the
ray containing $(\cos\phio,\sin\phio)$. The cylinder has unit radius.} \label{fig1}
\end{figure}
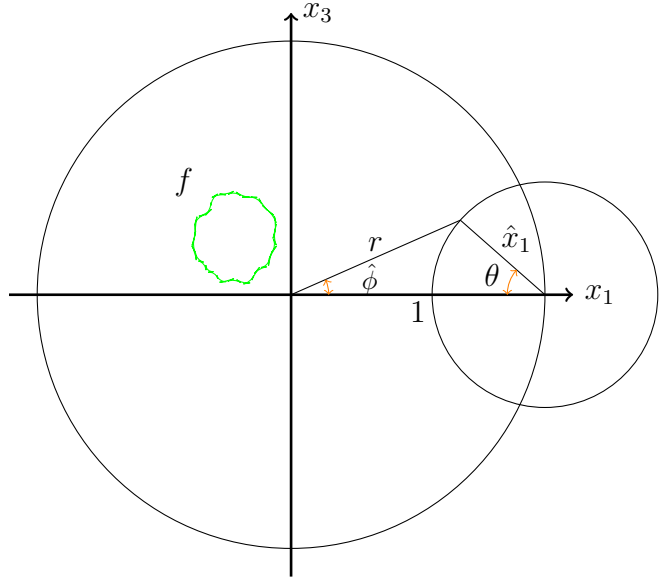

\begin{figure}[!h]
\centering
\begin{tikzpicture}[scale=0.75]
\draw [->,line width=1pt] (-5,0)--(5,0);
\draw [->,line width=1pt] (0,-5)--(0,5)node[right] {$\hat{x}_2$};
\draw[green,rounded corners=1mm] (0-1,2.6458) \irregularcircle{0.75cm}{1mm};
\node at (-0.9-1,2.6458+1) {$f$};
\coordinate (x) at (3,2.6458);
\coordinate (O) at (0,0);
\coordinate (c) at (4.5,0);
\draw (c)--(x);
\draw (x)--(4.5,2.6458);
\node at (3.9,3) {$\hat{x}_1$};
\node at (4.85,1.32) {$\hat{x}_2$};
\node at (-2.25,-0.3) {1};
\draw (4.5,-5)--(4.5,5);
\draw (-4.5,-5)--(-4.5,5);
\draw (4.5,0) ellipse (2 and 4);
\node at (3.5,-0.3) {$p$};

\node at (4.7,-0.3) {$\vs$};

\draw (4.5,0) [line width = 2pt] circle (0.05);

\node at (4.8,-2) {$\frac{p}{s}$};
\draw [dashed] (-7,2.6458)node[left]{$(x_1,x_3)$ plane}--(7,2.6458);
\end{tikzpicture}
\caption{Cylindrical geometry $(\hat{x}_1,\hat{x}_2)$ plane cross
section. The $(x_1,x_3)$ plane of Figure \ref{fig1} is drawn as a dashed line.}
\label{fig2}
\end{figure}
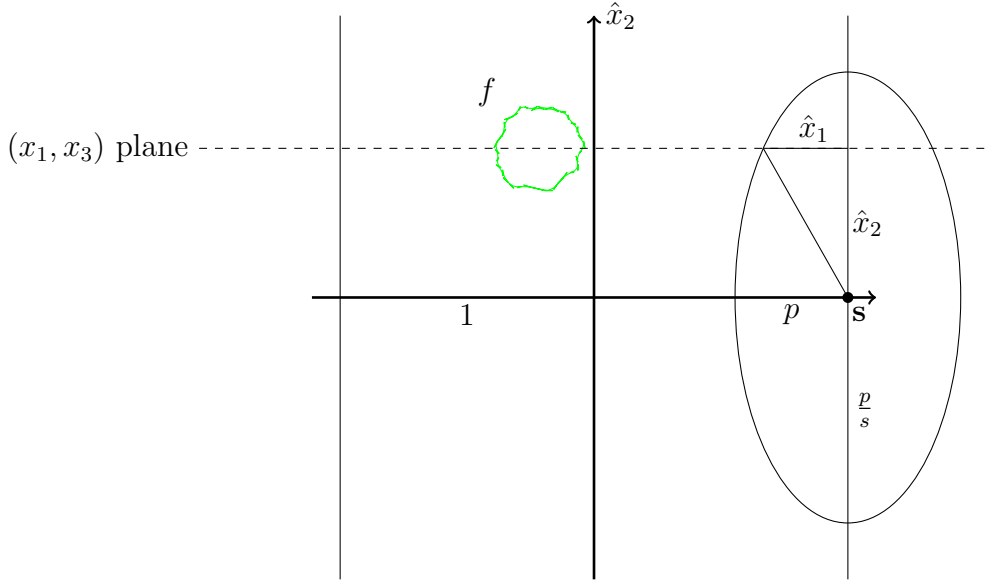

The  following proposition is the first step in writing $R_s f$ in
terms of Volterra equations of Fourier coefficients of $f$.

\begin{proposition}
Let $f\in L^2_c(\mathbb{R}^3)$. Then
\begin{equation}
R_sf(p,\phi_0,y_0)=\int_{-\frac{p}{s}}^{\frac{p}{s}}\sqrt{p^2-s^2\hat{x}_2^2+s^4\hat{x}_2^2}\int_{-\pi}^{\pi}f\paren{r,\sin^{-1}\paren{\frac{\hat{x}_1}{r}\sin\theta}+\phi_0,\hat{x}_2+y_0}\mathrm{d}\theta\mathrm{d}\hat{x}_2.
\end{equation}
\end{proposition}
\begin{proof}
Let
$\mathrm{d}l=\sqrt{1+\paren{\frac{\mathrm{d}\hat{x}_1}{\mathrm{d}\hat{x}_2}}^2}\mathrm{d}\hat{x}_2$
be the arc measure on the ellipse in Figure \ref{fig2}. Then the
surface element on the spheroid of revolution is
\begin{equation}
\mathrm{d}A=\hat{x}_1\mathrm{d}l\mathrm{d}\theta
=\sqrt{p^2-s^2\hat{x}_2^2+s^4\hat{x}_2^2}\mathrm{d}\theta\mathrm{d}\hxt.
\end{equation}
Thus, using equation \eqref{cyl_param} and \eqref{cyl2} to rewrite
$\phi = \hphi+\phio$ in terms of $\theta$, it follows that
\begin{equation}
\label{radon_s}
R_sf(p,\phi_0,y_0)=\int_{-\frac{p}{s}}^{\frac{p}{s}}\sqrt{p^2-s^2\hat{{x}}_2^2+s^4\hat{{x}}_2^2}\int_{-\pi}^{\pi}f\paren{r,\sin^{-1}\paren{\frac{\hat{x}_1}{r}\sin\theta}+\phi_0,\hat{{x}}_2+y_0}\mathrm{d}\theta\mathrm{d}\hat{{x}}_2,
\end{equation}
which completes the proof.\end{proof}

We now have our main injectivity result.
\begin{theorem}
For any fixed $s\in (0,1]$, $R_s$ is injective on domain $L^2_c(C)$,
where $C$ is the open unit cylinder .

\end{theorem}
\begin{remark}
The following proof uses some similar intuitions to that of \cite{ambartsoumian2010inversion}, applied in that paper to circular Radon transforms. We extend such ideas to three-dimensions, and to spheroid surfaces.
\end{remark}
\begin{proof}  Let $\eps>0$. 
 We first prove the theorem if $f\in L^2_c(\Ceps)$, where
\bel{def:Ceps}\Ceps=\sparen{\vx\in\mathbb{R}^3 \st
\sqrt{x_1^2+x_3^2}<1-\epsilon}.\ee

Let $f\in L^2_c(\Ceps)$. Taking the Fourier transform in $y_0$ on both
sides on \eqref{radon_s} yields
\begin{equation}
\label{radon_s_1}\begin{split}
&\widehat{R_sf}(p,\phi_0,\eta)=\\
&2\int_{0}^{\frac{p}{s}}\cos(\eta
\hxt)\sqrt{p^2-s^2\hat{x}_2^2+s^4\hat{x}_2^2}\int_{-\pi}^{\pi}\hat{f}\paren{r,\sin^{-1}\paren{\frac{\hat{x}_1}{r}\sin\theta}+\phi_0,\eta}\mathrm{d}\theta\mathrm{d}\hat{x}_2,
\end{split}
\end{equation}
where $\eta$ is dual to $y_0$. We now calculate the
Fourier components in $\phi_0$ on both sides of \eqref{radon_s_1},
where
$\hat{f}_n\paren{r,\eta}=\frac{1}{2\pi}\int_{0}^{2\pi}\hat{f}(r,\phi,\eta)e^{-i
n \phi}\mathrm{d}\phi$:
\begin{equation}
\label{1232}
\widehat{R_sf}_n(p,\eta)=
4\int_{0}^{\frac{p}{s}}\cos(\eta
\hxt)\sqrt{p^2-s^2\hat{x}_2^2+s^4\hat{x}_2^2}\int_{0}^{\pi}\cos(n\hphi)\hat{f}_n\paren{r,\eta}\mathrm{d}\theta\mathrm{d}\hat{x}_2.
\end{equation}
Note that $r$ and $\hphi$ depend on $\theta$
and $\hxo$ as given in \eqref{cyl2}, with $r$ even as a function of $\theta$, and $\hat{\phi}$ odd as a function of $\theta$.

 Making the change of variables $\hat{x}_1=\sqrt{p^2-s^2\hat{x}_2^2}$ in the $\hat{x}_2$ integral yields
\begin{equation}\begin{split}
&\widehat{R_sf}_n(p,\eta)=\\
&\frac{4}{s}\int_{0}^{p}\frac{\hat{x}_1\sqrt{\hat{x}_1^2+s^2(p^2-\hat{x}_1^2)}}{\sqrt{p^2-\hat{x}_1^2}}\cos\paren{\frac{\eta}{s}\sqrt{p^2-\hat{x}_1^2}}\int_{0}^{\pi}\cos(n\hphi)\hat{f}_n\paren{r,\eta}\mathrm{d}\theta\mathrm{d}\hat{x}_1.
\end{split}\end{equation}
Let us now do a change of variables in the $\theta$ integral.
Using Figure \ref{fig1} and \eqref{cyl2}, we see
$r=\sqrt{\hat{x}_1^2+1-2\hat{x}_1\cos\theta}$, and
$\sin\hphi=\frac{\hat{x}_1}{r}\sin\theta$. We have
$$\frac{\mathrm{d}r}{\mathrm{d}\theta}=\frac{\hat{x}_1}{r}\sin\theta=\sin\hphi,\ \ \  \cos\hphi=\frac{r^2+1-\hat{x}_1^2}{2r}.$$
Then
\small
\begin{equation}\label{radon_s_2_1}\begin{split}
&\widehat{R_sf}_n(p,\eta)=\\
&\frac{4}{s}\int_{0}^{p}\frac{\hat{x}_1\sqrt{\hat{x}_1^2+s^2(p^2-\hat{x}_1^2)}}{\sqrt{p^2-\hat{x}_1^2}}\cos\paren{\frac{\eta}{s}\sqrt{p^2-\hat{x}_1^2}}\int_{1-\hat{x}_1}^{1-\epsilon}\frac{T_{|n|}\paren{\frac{r^2+1-\hat{x}_1^2}{2r}}}{\sqrt{1-\paren{\frac{r^2+1-\hat{x}_1^2}{2r}}^2}}\hat{f}_n\paren{r,\eta}\mathrm{d}r\mathrm{d}\hat{x}_1\end{split}\end{equation}\normalsize
where $T_{|n|}$ is a Chebyshev polynomial degree $|n|$,
$n\in\mathbb{Z}$. Because $f$ is supported in $C_\eps$,  the upper limit in the  inner
integral in \eqref{radon_s_2_1} can be  $r=1-\eps$, rather than
$r=1+\hxo$.  Now, using Fubini's theorem, we
see\small\begin{equation}\label{radon_s_2_2}\begin{split}
&\widehat{R_sf}_n(p,\eta)=\\
&\frac{4}{s}\int_{1-p}^{1-\epsilon}\int_{1-r}^{p}\frac{\hat{x}_1\sqrt{\hat{x}_1^2+s^2(p^2-\hat{x}_1^2)}\cos\paren{\frac{\eta}{s}\sqrt{p^2-\hat{x}_1^2}}T_{|n|}\paren{\frac{r^2+1-\hat{x}_1^2}{2r}}}{\sqrt{p^2-\hat{x}_1^2}\sqrt{1-\paren{\frac{r^2+1-\hat{x}_1^2}{2r}}^2}}\hat{f}_n\paren{r,\eta}\mathrm{d}\hat{x}_1\mathrm{d}r,
\end{split}
\end{equation}
\normalsize

 Substituting $u=1-r$ yields
\begin{equation}
\label{volt2}
\begin{split}
\widehat{R_sf}_n(p,\eta)&=\frac{4}{s}\int_{0}^{p}K_n(\eta ; p,u)\tilde{\hat{f}}_n\paren{u,\eta}\mathrm{d}u,
\end{split}
\end{equation}
a Volterra equation of the first kind, where 
\begin{equation}\label{kernel:Kn}
K_n(\eta ; p,u)=\int_{u}^{p}\frac{\hat{x}_1\sqrt{\hat{x}_1^2+s^2(p^2-t^2)}\cos\paren{\frac{\eta}{s}\sqrt{p^2-\hat{x}_1^2}}T_{|n|}\paren{\frac{(1-u)^2+1-\hat{x}_1^2}{2(1-u)}}}{\sqrt{p^2-\hat{x}_1^2}\sqrt{1-\paren{\frac{(1-u)^2+1-\hat{x}_1^2}{2(1-u)}}^2}}\mathrm{d}\hat{x}_1,
\end{equation}
and $\tilde{\hat{f}}_n\paren{u,\eta}=\hat{f}_n\paren{1-u,\eta}$.

To show injectivity, we let $\eps_1\in (0,1/2)$ and first bound the
kernel $K_n$ and its derivative for each fixed $\eta$ on the set
\bel{def:Deps} D_{\eps_1} = \sparen{(p,u)\st p\in
[\eps,1-\eps_1],\ u\in  [\eps,
p]}.\ee We have
\begin{equation}
\begin{split}
\sqrt{1-\paren{\frac{(1-u)^2+1-\hat{x}_1^2}{2(1-u)}}^2}&=\sqrt{1-\paren{1+\frac{u^2-\hat{x}_1^2}{2(1-u)}}^2}\\
&=\sqrt{\frac{\hat{x}_1^2-u^2}{1-u}+\frac{(\hat{x}_1^2-u^2)(u^2-\hat{x}_1^2)}{4(1-u)^2}}\\
&=\frac{\sqrt{\hat{x}_1^2-u^2}}{\sqrt{1-u}}\sqrt{1+\frac{u^2-\hat{x}_1^2}{4(1-u)}}.
\end{split}
\end{equation}
Thus
\begin{equation}
\label{B.9}
\begin{split}
&K_n(\eta ; p,u)=
\sqrt{1-u}\int_{u}^{p}\frac{\hat{x}_1\sqrt{\hat{x}_1^2+s^2(p^2-\hat{x}_1^2)}\cos\paren{\frac{\eta}{s}\sqrt{p^2-\hat{x}_1^2}}T_{|n|}\paren{\frac{(1-u)^2+1-\hat{x}_1^2}{2(1-u)}}}{\sqrt{p^2-\hat{x}_1^2}\sqrt{\hat{x}_1^2-u^2}\sqrt{1+\frac{u^2-\hat{x}_1^2}{4(1-u)}}}\mathrm{d}\hat{x}_1\\
&=\sqrt{1-u}\int_{0}^{1}\frac{\hat{x}_1\sqrt{\hat{x}_1^2+s^2(p^2-\hat{x}_1^2)}\cos\paren{\frac{\eta}{s}\sqrt{p^2-\hat{x}_1^2}}T_{|n|}\paren{\frac{(1-u)^2+1-\hat{x}_1^2}{2(1-u)}}}{\sqrt{v}\sqrt{1-v}\sqrt{p+\hat{x}_1}\sqrt{\hat{x}_1+u}\sqrt{1+\frac{u^2-\hat{x}_1^2}{4(1-u)}}}\mid_{\hat{x}_1=u+v(p-u)}\mathrm{d}v,
\end{split}
\end{equation}
after substituting \bel{v change}\hat{x}_1=u+v(p-u)\ee in the last step. We have
\begin{equation}
\begin{split}
K_n(\eta ; p,p)&=\frac{p\sqrt{1-p}}{2}\int_{0}^{1}\frac{1}{\sqrt{v}\sqrt{1-v}}\mathrm{d}v\\
&=\frac{\pi  p\sqrt{1-p}}{2},
\end{split}
\end{equation}
and thus $K_n(\eta,\cdot,\cdot)$ is non-zero on the diagonal, unless
$p=0,1$, for all $\eta\in\mathbb{R}$ and $n\in\mathbb{Z}$. The
support of $f$ is bounded away from the cylinder surface and we are
considering  $p\in [\eps,1-\eps_1]$, and thus we do not consider $p=0,1$.

We will now show that $K_n(\eta;p,u)$ and
$\frac{\mathrm{d}}{\mathrm{d}p}K_n(\eta ; p,u)$ are bounded for each
$\eta$ and all $(p,u)\in D_{\eps_1}$. To do this, we show that all the
terms dependent on $p$ under the integral on the second line of
\eqref{B.9} are bounded and have bounded first order derivative with
respect to $p$. First we have $|\hat{x}_1|\leq 1$ and from the change
of variable \eqref{v change}, $\frac{\mathrm{d}}{\mathrm{d}p}
\hat{x}_1 =v\leq 1$. Now
$$\sqrt{\hat{x}_1^2+s^2(p^2-\hat{x}_1^2)}=\sqrt{(1-s^2)\hat{x}_1^2+s^2p^2}\leq\sqrt{(1-s^2)+s^2}=1,$$
and
$$\frac{\mathrm{d}}{\mathrm{d}p}\sqrt{(1-s^2)\hat{x}_1^2+s^2p^2}=\frac{ps^2+v(1-s^2)\hat{x}_1}{\sqrt{(1-s^2)\hat{x}_1^2+s^2p^2}}\leq \frac{1}{\sqrt{(1-s^2)\epsilon^2+s^2\epsilon^2}}= \frac{1}{\epsilon},$$
noting $\hat{x}_1\geq u\geq \epsilon$ and $p>\eps$. We have $\left|\cos\paren{\frac{\eta}{s}\sqrt{p^2-\hat{x}_1^2}}\right|\leq 1$, and
\begin{equation}\label{deriv-estimate}
\begin{split}
\frac{\mathrm{d}}{\mathrm{d}p}\cos\paren{\frac{\eta}{s}\sqrt{p^2-\hat{x}_1^2}}&=-\frac{\frac{\eta}{s}(p-v\hat{x}_1)}{\sqrt{p^2-\hat{x}_1^2}}\sin\paren{\frac{\eta}{s}\sqrt{p^2-\hat{x}_1^2}}\\
&=-\paren{\frac{\eta}{s}}^2(p-v\hat{x}_1)\sinc\paren{\frac{\eta}{s}\sqrt{p^2-\hat{x}_1^2}},
\end{split}
\end{equation}
and thus $\left|\frac{\mathrm{d}}{\mathrm{d}p}\cos\paren{\frac{\eta}{s}\sqrt{p^2-\hat{x}_1^2}}\right| \leq \paren{\frac{\eta}{s}}^2$. For $n\neq 0$, we have $\left|T_{|n|}\paren{\frac{(1-u)^2+1-\hat{x}_1^2}{2(1-u)}}\right|\leq 1$ and
\begin{equation}
\begin{split}
\frac{\mathrm{d}}{\mathrm{d}p}T_{|n|}\paren{\frac{(1-u)^2+1-\hat{x}_1^2}{2(1-u)}}&=-\frac{v\hat{x}_1}{1-u}T'_{|n|}\paren{\frac{(1-u)^2+1-\hat{x}_1^2}{2(1-u)}}\\
&=-\frac{|n|v\hat{x}_1}{1-u}U_{|n|-1}\paren{\frac{(1-u)^2+1-\hat{x}_1^2}{2(1-u)}},
\end{split}
\end{equation}
where $U_n$ is a Chebyshev polynomial of the second kind. It follows that
$$\left|\frac{\mathrm{d}}{\mathrm{d}p}T_{|n|}\paren{\frac{(1-u)^2+1-\hat{x}_1^2}{2(1-u)}}\right|\leq
\frac{|n|^2}{\epsilon_1},$$
using Figure \ref{fig1} and the Law of Cosines,  that
$|U_{|n|}(x)|\leq |n|+1$ for $\abs{x}\leq 1$, and $\frac{1}{1-u}\leq
\frac{1}{\epsilon_1}$. The $n=0$ case is trivial since $T_0=1$. 



Now, we have $\frac{1}{\sqrt{p+\hat{x}_1}\sqrt{\hat{x}_1+u}}\leq\frac{1}{2\epsilon}$ and
$$\frac{\mathrm{d}}{\mathrm{d}p}\paren{\frac{1}{\sqrt{p+\hat{x}_1}\sqrt{\hat{x}_1+u}}}=\frac{1+2v}{\sqrt{p+\hat{x}_1}\sqrt{\hat{x}_1+u}}\leq
\frac{3}{2\epsilon}.$$
Finally, we have
$$\frac{u^2-\hat{x}_1^2}{4(1-u)}\geq \frac{u^2-1}{4(1-u)}=-\frac{(u+1)}{4}\geq -\frac{1}{2}.$$
Thus, $\paren{1+\frac{u^2-\hat{x}_1^2}{4(1-u)}}^{-\frac{1}{2}}\leq \sqrt{2}$, and
$$\frac{\mathrm{d}}{\mathrm{d}p}\paren{1+\frac{u^2-\hat{x}_1^2}{4(1-u)}}^{-\frac{1}{2}}=\frac{v\hat{x}_1}{4(1-u)}\paren{1+\frac{u^2-\hat{x}_1^2}{4(1-u)}}^{-\frac{3}{2}}\leq \frac{1}{\epsilon_1\sqrt{2}}.$$

After putting all this together, we can convert \eqref{volt2} into a
Volterra equation of the second kind with bounded kernel for $(p,u)\in
D_{\eps_1}$ and invert by successive approximations using classical
Volterra Integral equation theory \cite{Tricomi, Q1983-rotation}.
Therefore, if $R_sf=0$, then \eqref{volt2} implies that
$\hat{f}_n(r,\eta)=0$ for all $\eta$ and for $r\in [\eps_1,1-\eps]$. The
lower limit of $\eps_1$ was used only to show $K_n$ is bounded and so
$\hat{f}_n(r,\eta) = 0$ for all $r\in [0,1-\eps]$.

Thus, any $f\in L^2_c(C)$ is uniquely
determined by $R_sf$, for any fixed $s\in (0,1]$.
This implies that $R_s$ is injective on $L_c^2(C)$.
\end{proof}

\begin{remark}Note that the proof  requires
$f$ to be zero near the boundary of the cylinder. This is needed so
that the inner integral in \eqref{radon_s_2_1} can have upper limit
$1-\eps$ instead of $1$.  This allows us to prove that $K$ satisfies
the hypotheses to make \eqref{volt2} an invertible Volterra equation.

The standard inversion result for Volterra equations would not apply
to $K_n$ if that upper limit were $1$ since $K_n$ would no longer be
bounded. This suggests there could be a null space for $R_s$ on
$L^2(C)$.
\end{remark}

We now discuss the visible singularities.

\subsection{Visible singularities} \label{visible-cylinder:sect} In
this section, we investigate the singularity coverage (or edge
detection) using spheroid and spherical integral data when the surface
of sources and receivers is a unit cylinder with central axis $x_2$,
as considered in the previous section on injectivity. Let $\partial C$
denote the cylinder of source and receiver positions. We consider
$\phi$ and $x_2$ in the range $\phi \in [0,2\pi]$ and $x_2\in [-1,1]$,
and parameterize  $\partial C$ using cylindrical coordinates
$\vs(\phi,x_2) = (\cos\phi,x_2,\sin\phi)\in \partial C$, as in
\eqref{def:s} 
For every $\vx\in\{\sqrt{x_1^2+x_3^2}<1\} = C_I$ (i.e., for every $\vx$ in the interior of $\partial C$), we can calculate the proportion of wavefront directions that are detected by spherical and spheroid integral data. The spherical data is three-dimensional, and the degrees of freedom are $(p,\phi_0,y_0)$. The spheroid data is four-dimensional, with degrees of freedom $(s,p,\phi_0,y_0)$. We wish to investigate whether the additional scaling factor, $s$, offers any improvement in terms of edge detection. We discretize $\partial C$ with $\phi$ at $1^{\circ}$ intervals, $\phi\in \{\frac{j \pi}{180} : 0\leq j\leq 179\}=\Phi$, and $x_2\in\{-1+\frac{2j}{N-1} : 0\leq j\leq N-1\}=X_2$, where $N\geq 1$ controls the level of discretization along the $x_2$ axis. For the spherical data, we consider all spheres with centers $c\in \{(\cos\phi,x_2,\sin\phi) :\phi\in\Phi, x_2\in X_2\}=\partial C_0$. We consider all spheroids with axis of revolution parallel to $x_2$, whose foci $c_1,c_2 \in \partial C_0$. For every given $c$ and $\vx \in C_I$, we calculate the wavefront direction, $\xi = \frac{\vx- c}{\|\vx- c\|}$, detected. For every $\vx \in C_I$, we calculate all $180 \times N$ wavefront directions detected at $\vx$, and use this information to build a 3-D map of the total wavefront detection on $C_I$. Similarly, we can calculate a 3-D map of the directional coverage for the spheroid integral data, and compare against the spherical map.
\begin{figure}[!h]
\centering
\begin{subfigure}{0.32\textwidth}
\includegraphics[width=0.9\linewidth, height=4cm, keepaspectratio]{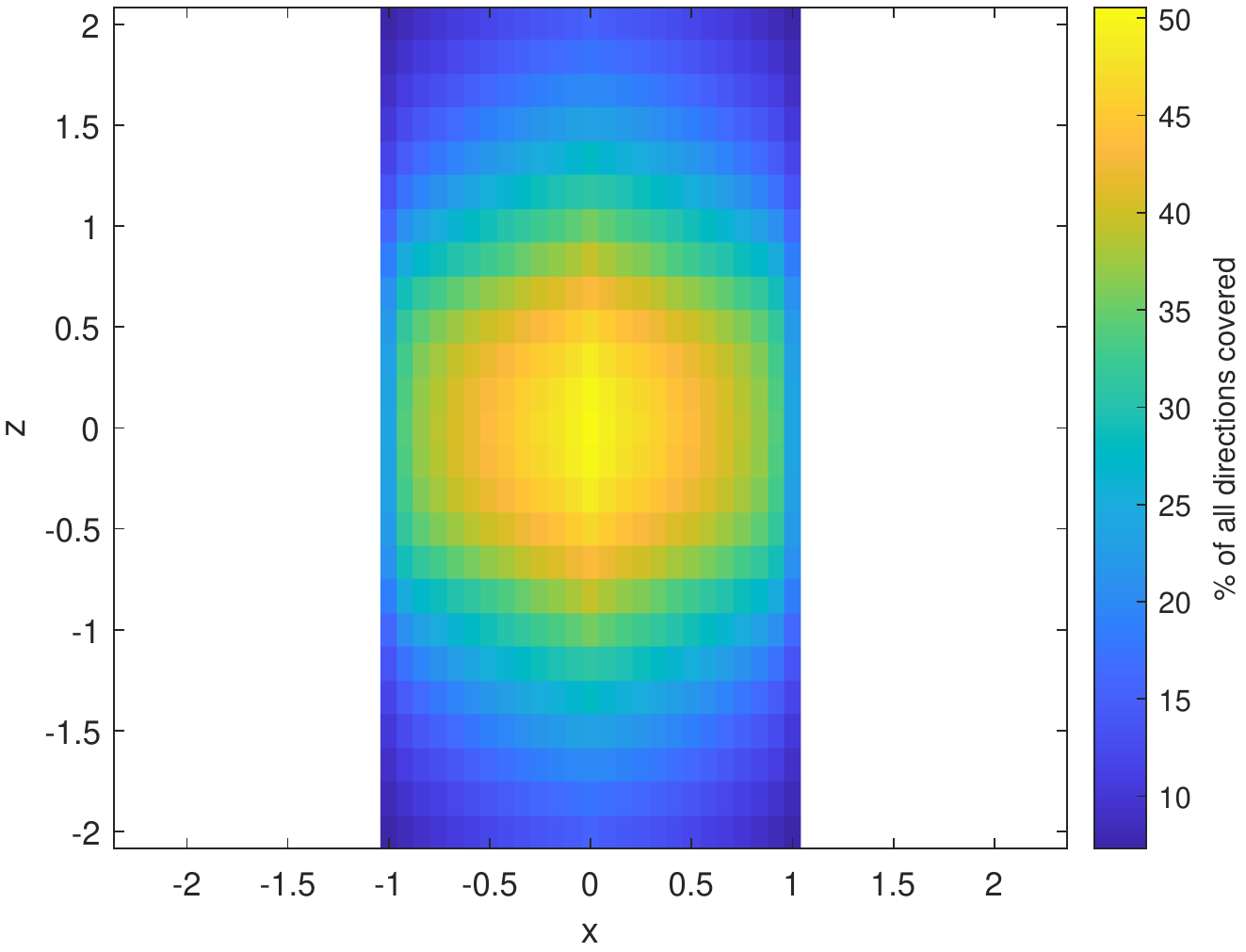}
\end{subfigure}
\begin{subfigure}{0.32\textwidth}
\includegraphics[width=0.9\linewidth, height=4cm, keepaspectratio]{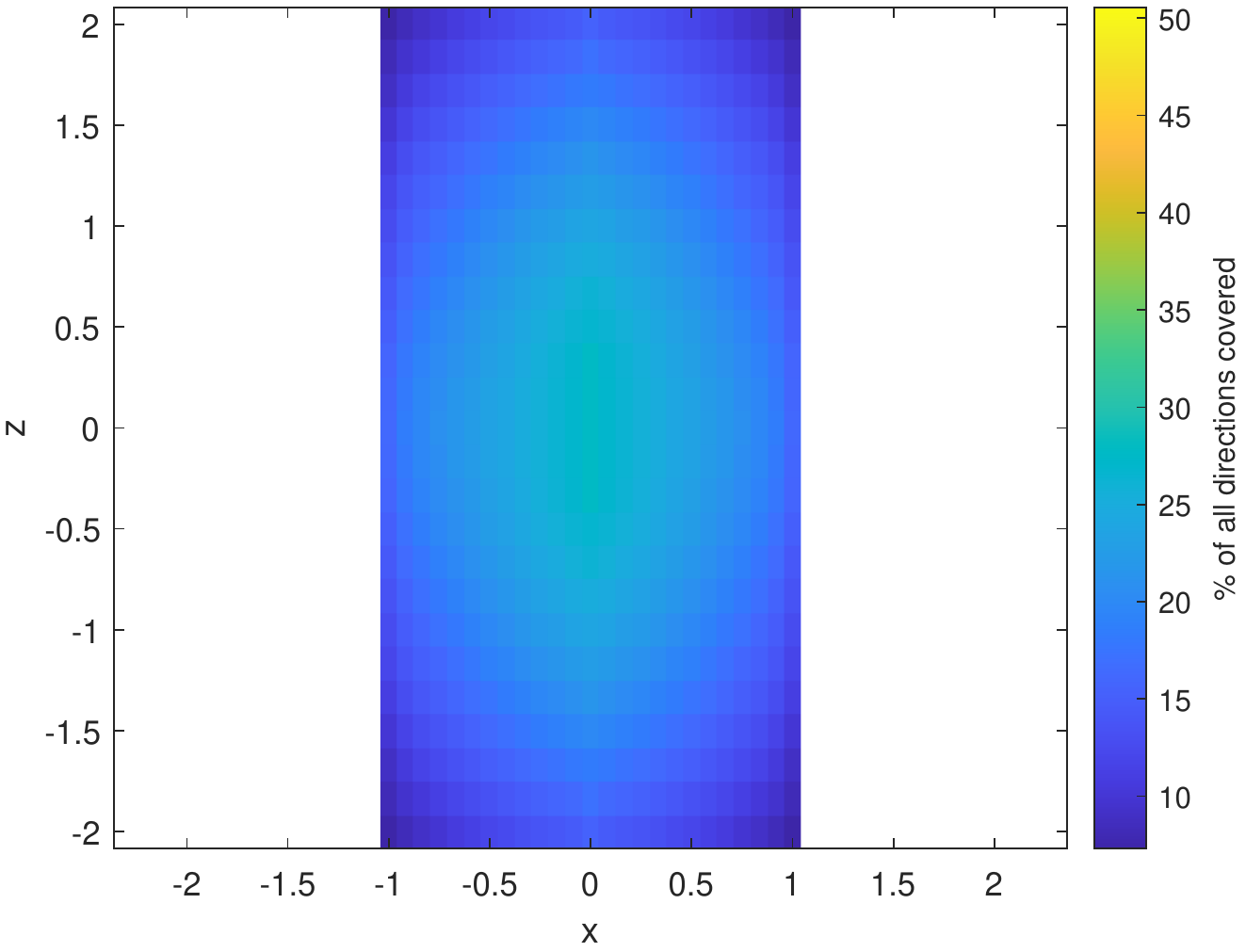} 
\end{subfigure}
\begin{subfigure}{0.32\textwidth}
\includegraphics[width=0.9\linewidth, height=4cm, keepaspectratio]{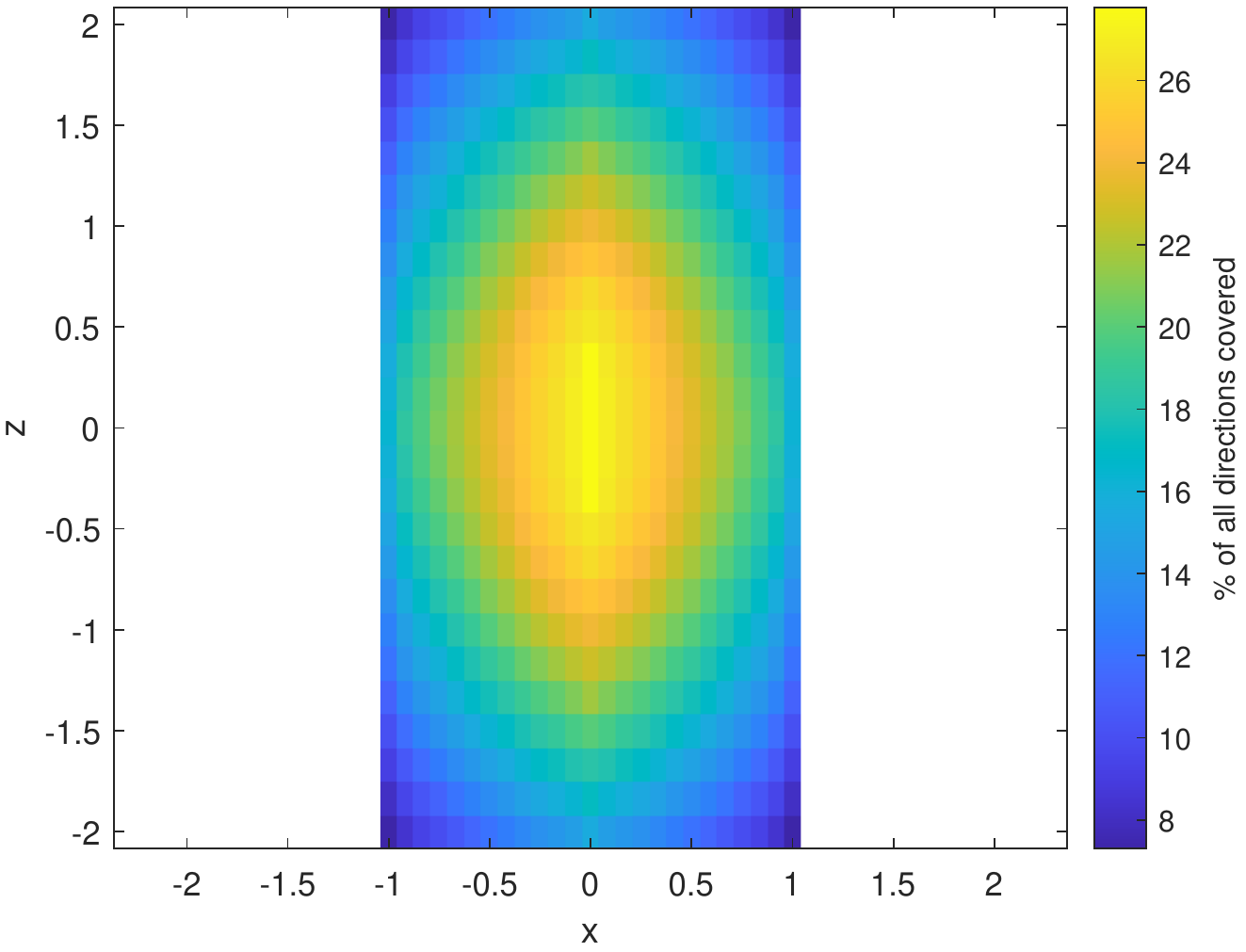}
\end{subfigure}
\begin{subfigure}{0.32\textwidth}
\includegraphics[width=0.9\linewidth, height=4cm, keepaspectratio]{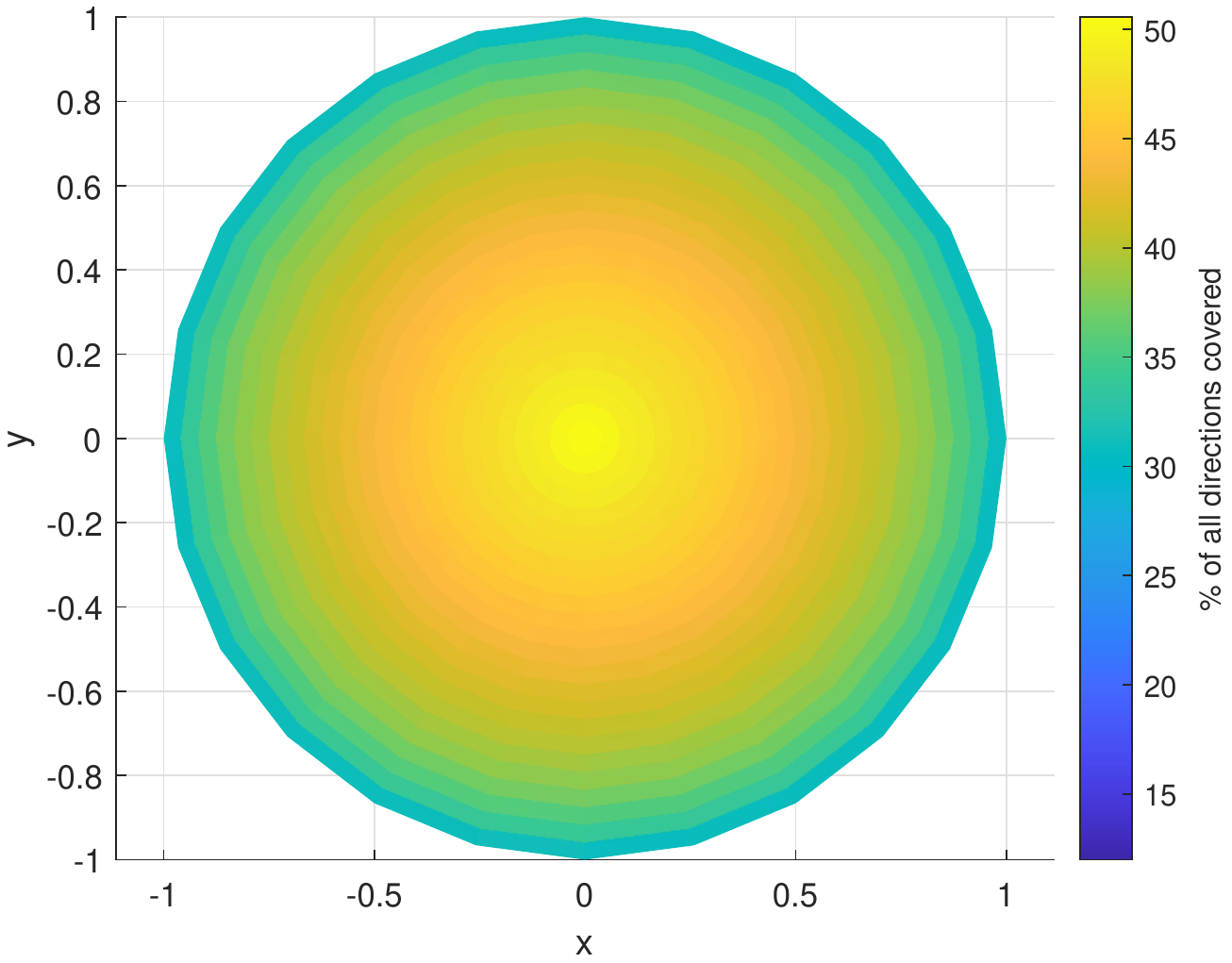}
\subcaption*{4D spheroid data} \label{Fvis_2d}
\end{subfigure}
\begin{subfigure}{0.32\textwidth}
\includegraphics[width=0.9\linewidth, height=4cm, keepaspectratio]{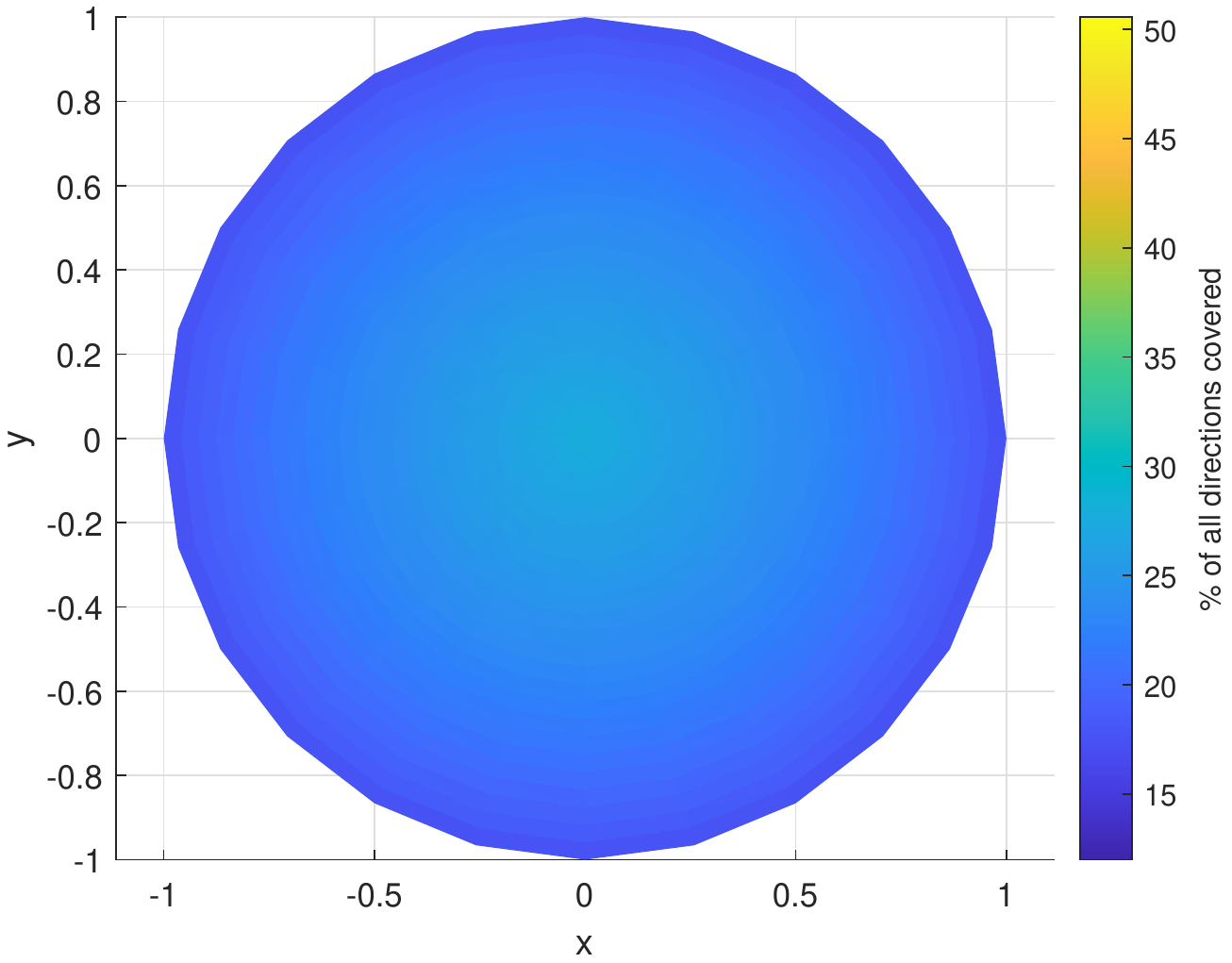} 
\subcaption*{3D sphere data} \label{Fvis_2e}
\end{subfigure}
\begin{subfigure}{0.32\textwidth}
\includegraphics[width=0.9\linewidth, height=4cm, keepaspectratio]{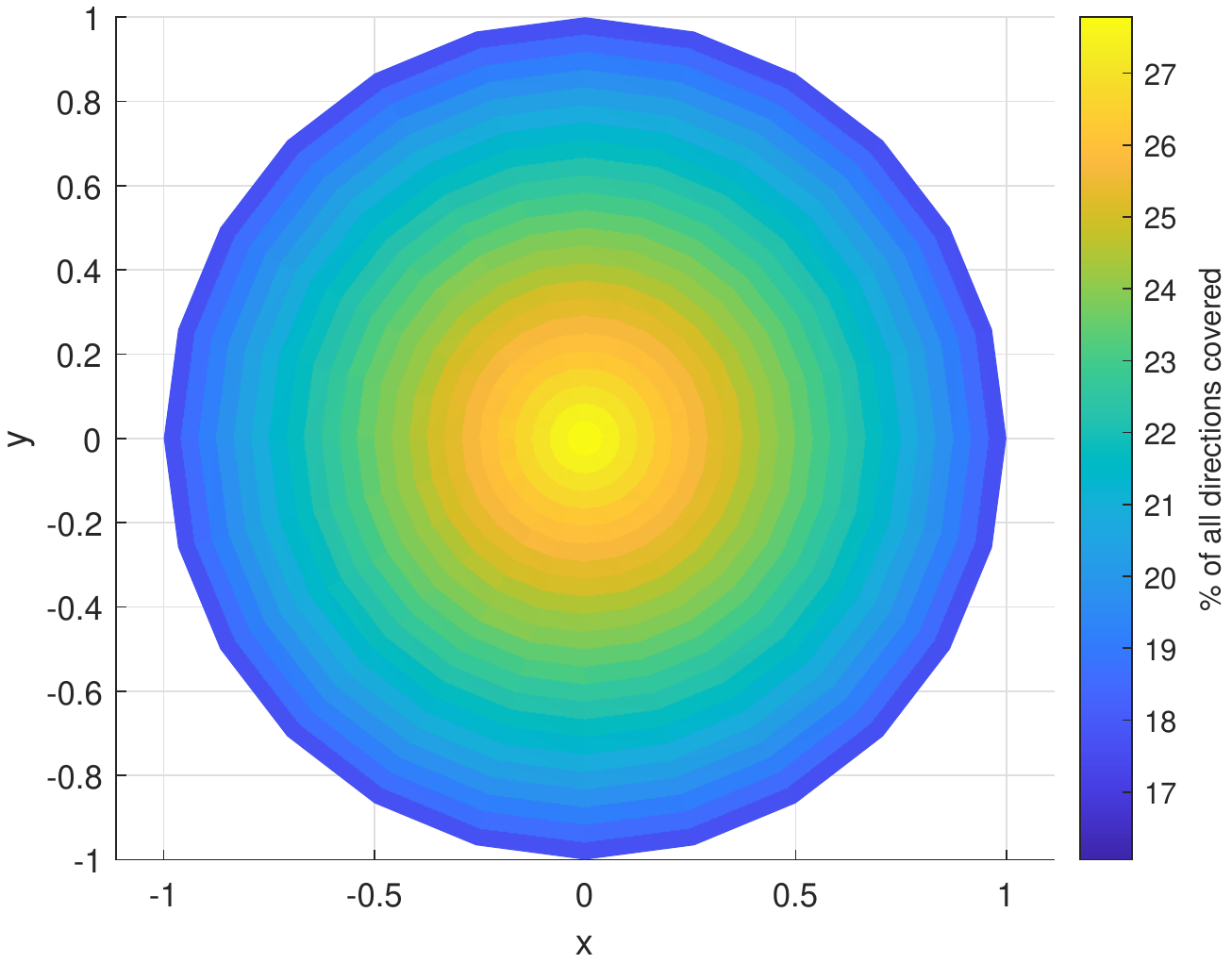}
\subcaption*{3D sphere data - cropped color bar} \label{Fvis_2f}
\end{subfigure}
\caption{Detectable singularities when $N=50$. Top row - $(x_1,x_2)$ plane. Bottom row - $(x_1,x_3)$ plane. Left - 4D spheroid data. Middle - 3D sphere data. Right - 3D sphere data with cropped color bar to better show the details.}
\label{Fvis_2}
\end{figure}

In Figure \ref{Fvis_2}, we present $(x_1,x_3)$ and $(x_1,x_2)$ plane cross-sections showing the directional coverage using spherical and spheroid data when $N=50$. 
The left-hand and middle columns of Figure \ref{Fvis_2} compare the directional coverage of spheroid and spherical data on the same scale. The right-hand column of Figure \ref{Fvis_2} shows the spherical wavefront detection with the color bar cropped so that the reader can better see the details. We can see that the wavefront coverage is significantly stronger using spheroid integral data, when compared to spherical, and thus the additional degree of freedom, $s$, has proven beneficial. 

To show what happens as the number of emitters, and level of $x_2$ discretization ($N$), varies, we plot curves of the average directional coverage over the $(x_2,x_3)$ and $(x_1,x_3)$ planes for varying $N$ in Figure \ref{Fvis_1}. For all $N\leq 100$, spheroid data offers greater average wavefront detection, when compared to spherical, for all $N\leq 100$, although the difference becomes less pronounced with increasing $N$.
\begin{figure}[!h]
\centering
\begin{subfigure}{0.32\textwidth}
\includegraphics[width=0.9\linewidth, height=4cm, keepaspectratio]{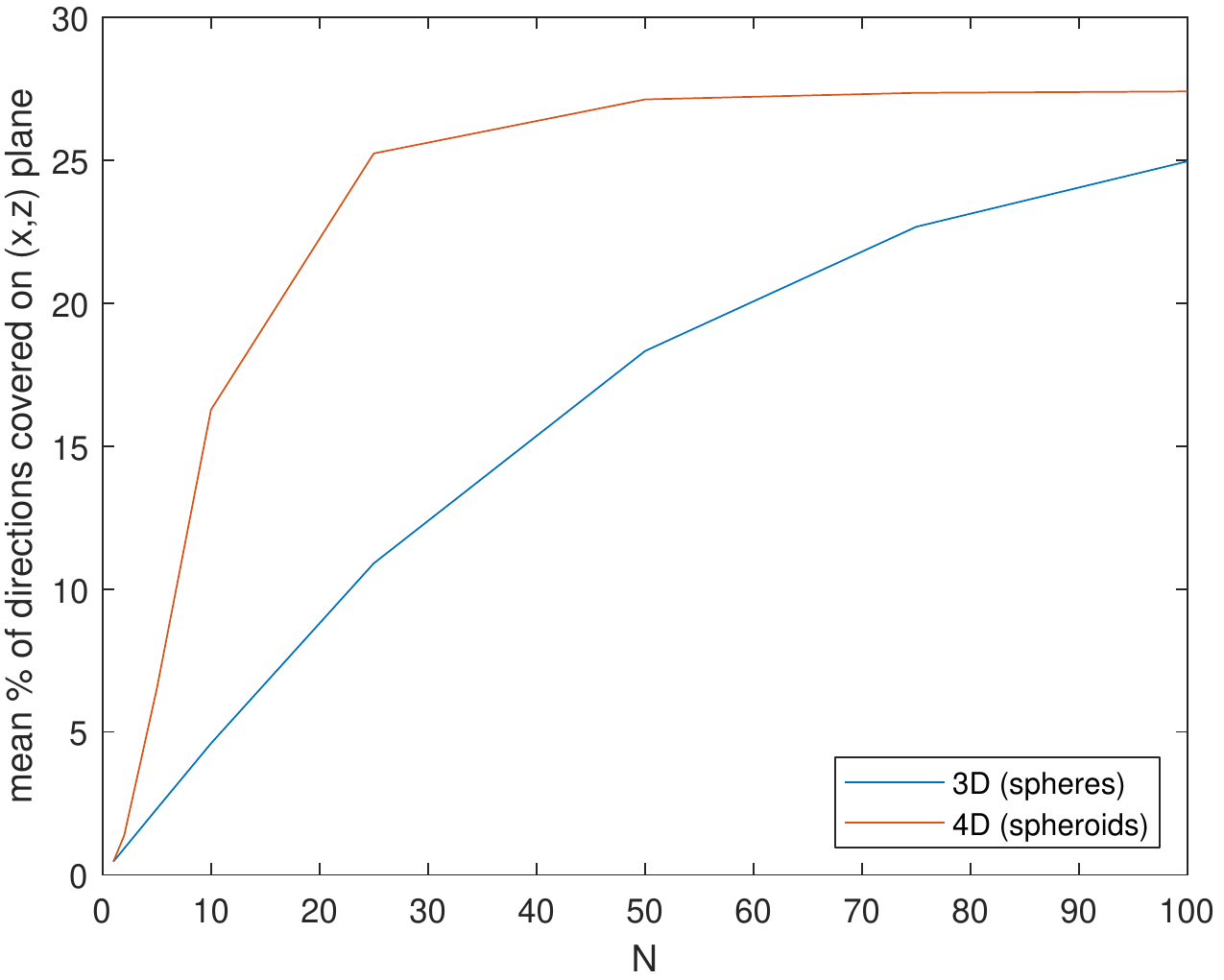}
\subcaption{$(x_1,x_2)$ plane} \label{Fvis_1a}
\end{subfigure}
\begin{subfigure}{0.32\textwidth}
\includegraphics[width=0.9\linewidth, height=4cm, keepaspectratio]{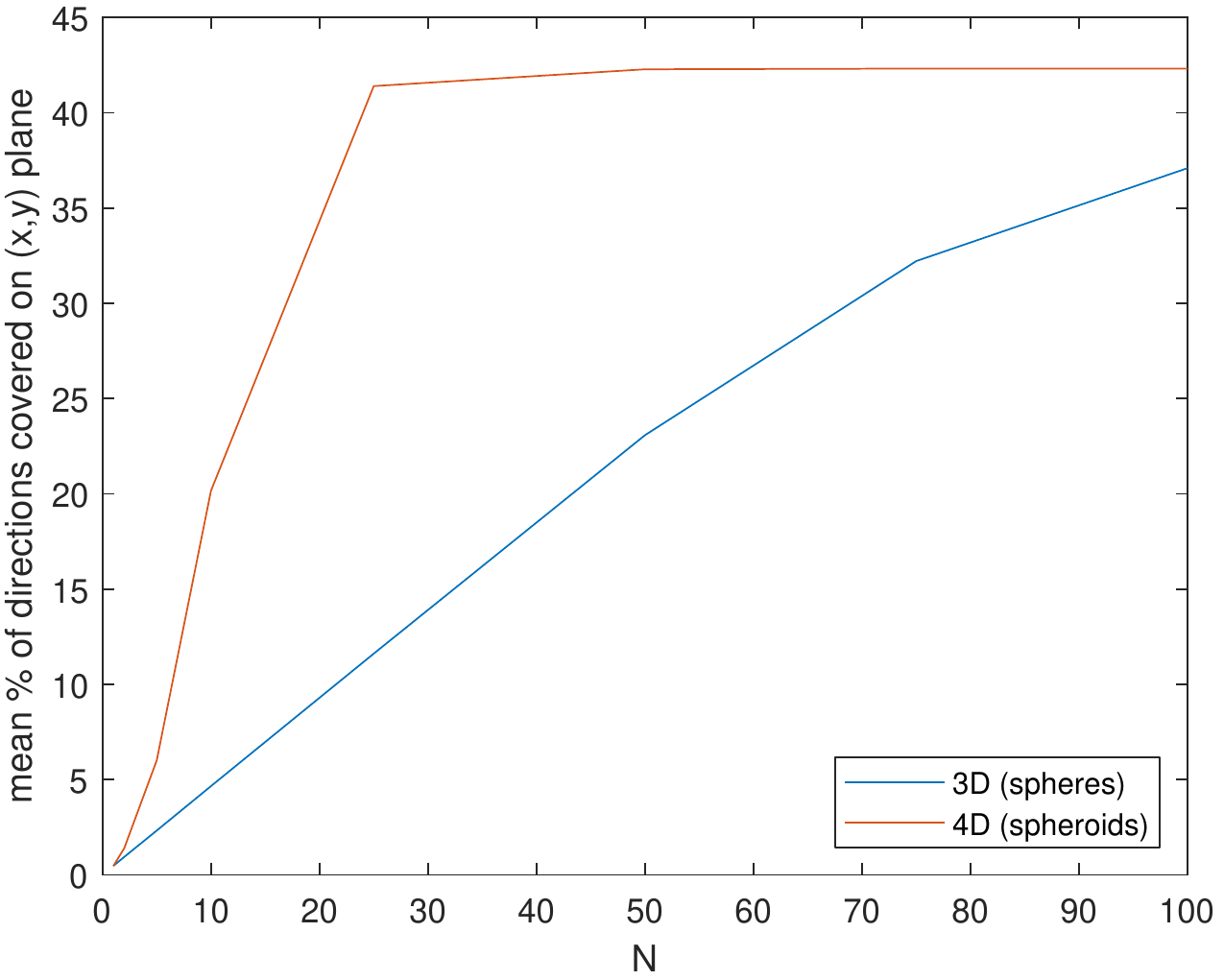} 
\subcaption{$(x_1,x_3)$ plane} \label{Fvis_1b}
\end{subfigure}
\caption{Mean directional coverage percentage over $(x_1,x_2)$ and $(x_1,x_3)$ plane cylinder cross-sections for varying $N$.}
\label{Fvis_1}
\end{figure}

Let $R_E(s,p,\phi_0,y_0) =  R_s(p,\phi_0,y_0)$.  If one were to design a URT scanner with a cylindrical set of emitters/receivers, as described in this section, it would be beneficial to measure spheroid integral data, $R_Ef$, over $R_1 f$ (spherical) data, as $R_E$ offers greater edge detection, especially with more limited emitter/receiver discretization. $R_E$ and $R_1$ both have the theoretical guarantees of injectivity and satisfaction of Bolker, as proven by our microlocal and injectivity theorems.

%

\section{Example image reconstructions in two-dimensions} 
\label{recon:sect}

In this section, we present two-dimensional image reconstructions
from spherical (circular) integral data. We consider two scanning
curves, $S$, one which is non-convex, and one which is convex. We
verify Corollary \ref{cor:RA-convex} by comparing the artifacts in (unfiltered) backprojection reconstructions of delta functions to artifacts predicted by our theory. In addition, we also investigate techniques to suppress the image artifacts in the non-convex curves case. Specifically, we apply discrete solvers and a Total Variation (TV) regularizer with smoothing parameter chosen by cross validation.

We define the Radon transform
\begin{equation}
R_cf(r,y_1)=Rf\paren{(y_1,q(y_1))^T,I_{2\times 2},r^2},
\end{equation}
where $q$ defines the set of circle centers, $S=q(\mathbb{R})$, and $r$ is the circle radius. In this section, we present reconstructions of $f$ from $R_cf$ data. We consider two example $q$, one non-convex with $q(y_1)=a(y_1+100)(y_1-100)y_1^2+100$ and
$a=5/10^6$, and one convex, with $q(y_1)=\frac{x^2}{50}-100$. See Figure \ref{F5} for an illustration of both curves. In both cases,  $\text{supp}(f) \subset \{(x_1,x_2)\in\mathbb{R}^2 : x_2>q(x_1), {\abs{x_1}<100}\}$.
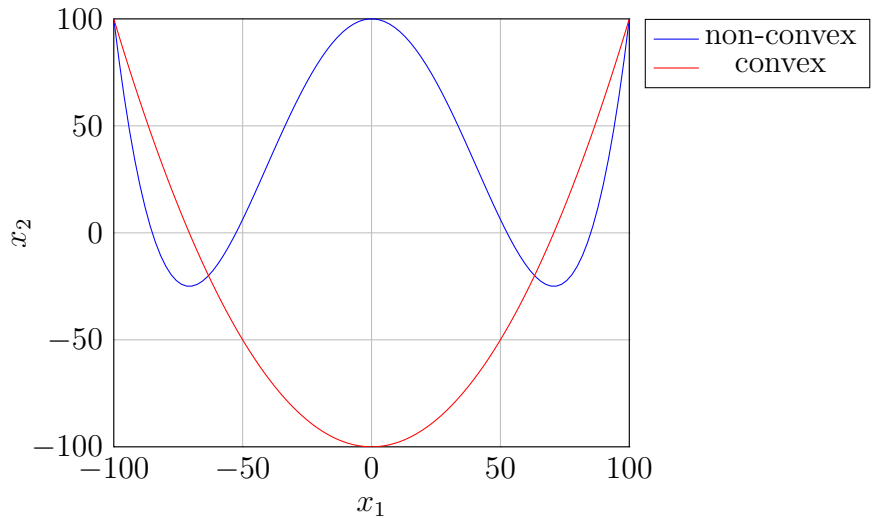
\begin{figure}[!h]
\centering
\begin{tikzpicture}[scale=1]
\begin{axis}[
  legend pos=outer north east ,xlabel=$x_1$, ylabel=$x_2$, xmin=-100,
xmax=100, ymin= -100,ymax=100, 
major tick length=0mm, grid=major
]
\addplot+[mark=none,samples=100,domain=-100:100] {(5/1000000)*(x+100)*(x-100)*x^2+100};
\addplot+[mark=none,samples=100,domain=-100:100] {(1/50)*x^2-100};
\legend{non-convex, convex}
\end{axis}
\end{tikzpicture}
\caption{Convex and non-convex measurement curve examples.} \label{F5}
\end{figure}
These example curves and function supports are chosen for two reasons. First, for $f\in
L^2_c(D)$, $R_cf$ uniquely determines $f$ for both $q$ considered, and hence there are no artifacts
due to a null space. This is true since $S$, for both $q$ in Figure \ref{F5}, is not the
union of a finite set and a Coxeter system of straight lines
\cite{agranovsky1996injectivity}. Second, $R_cf$ detects all
singularities in $D$. Thus, the only artifacts are due to noise and
the Bolker condition, which is our focus. We will also discover later, due to discretization limitations, streaking artifacts which occur along circles at the boundary of the data set. Similar boundary artifacts have been discussed previously in the literature, in regards to photo-acoustic tomography and sonar \cite{FrQu2015}.
\begin{figure}[!h]
\centering
\begin{subfigure}{0.32\textwidth}
\includegraphics[width=0.9\linewidth, height=0.9\linewidth, 
]{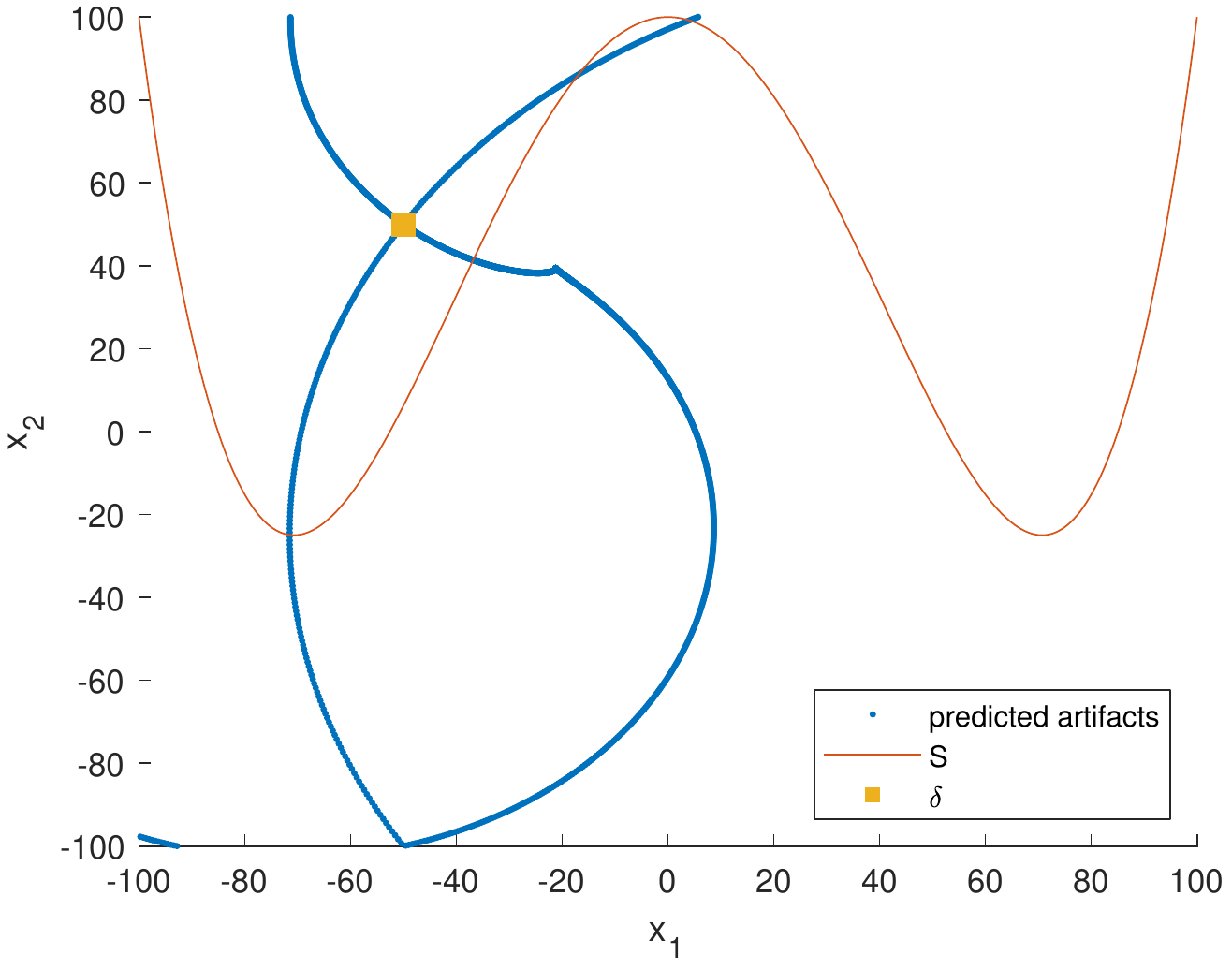}
\label{F6a}
\end{subfigure}
\begin{subfigure}{0.32\textwidth}
\includegraphics[width=0.9\linewidth,
height = 0.9\linewidth
]{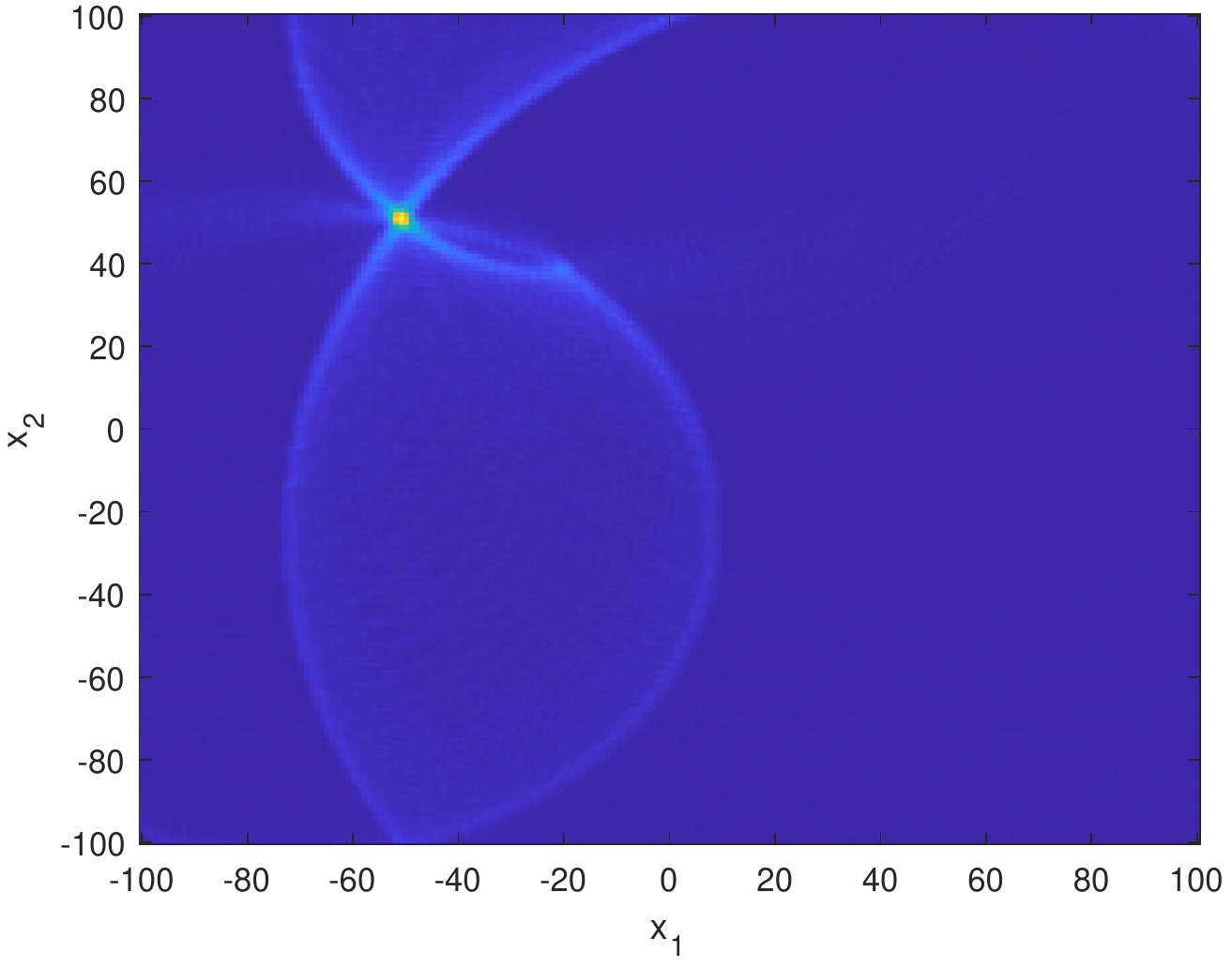} 
 \label{F6b}
\end{subfigure}

\begin{subfigure}{0.32\textwidth}
\includegraphics[width=0.9\linewidth,
height = 0.9\linewidth
]{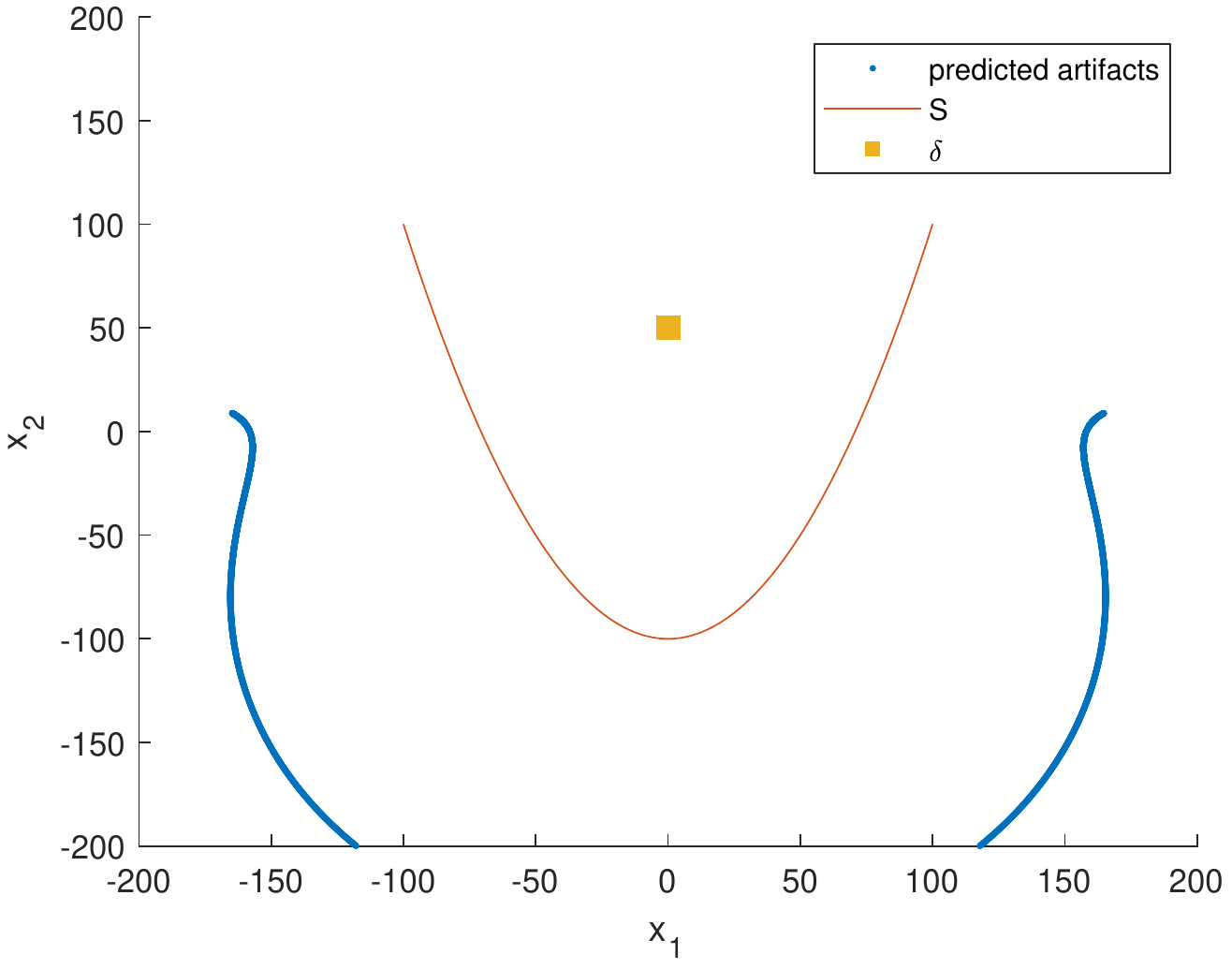} 
\subcaption*{predicted artifacts} \label{F6d}
\end{subfigure}
\begin{subfigure}{0.32\textwidth}
\includegraphics[width=0.9\linewidth,
height = 0.9\linewidth]{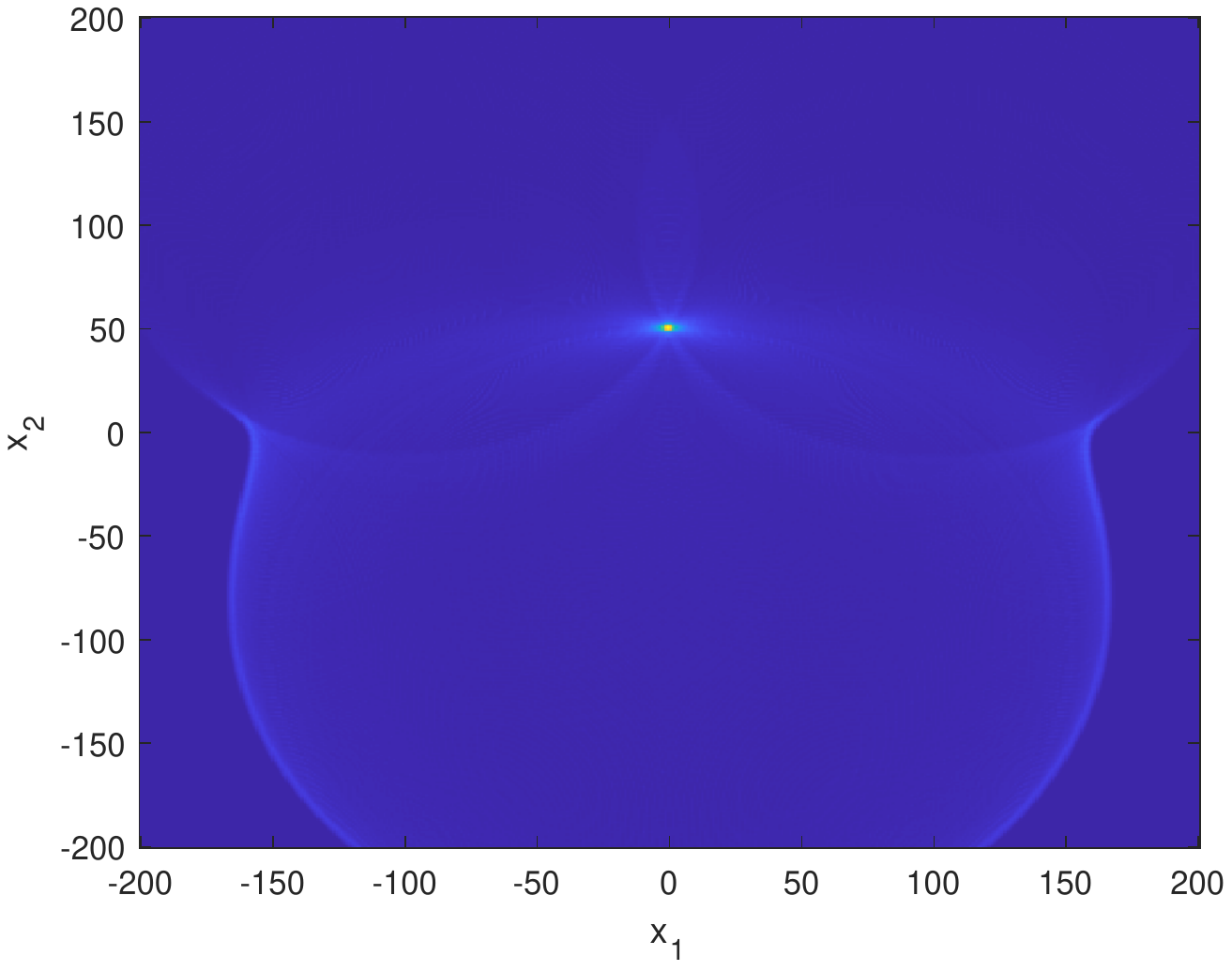}
\subcaption*{$A^TA\delta$ (observed artifacts)} \label{F6e}
\end{subfigure}
\caption{Predicted observed artifacts in delta function reconstructions. Top row - non-convex curve. Bottom row - convex curve.}
\label{F6}
\end{figure}

\subsection{Delta function reconstructions}
We now present unfiltered backprojection image reconstructions of
delta functions to validate the results of Theorem
\ref{thm:RA-general}. A delta function is supported at a single point
and has singularities (edges) in all directions. Thus, in a
reconstruction of a delta function, $\delta$, from circular integral
data with centers on $S=q(\mathbb{R})$, we would expect to see
artifacts which are the reflections of $\delta$ in planes tangent to
$S$. Let $A$ denote the discretized form of $R_c$. We sample circle
centers $c_i=\paren{-100+\frac{i-1}{2},q\paren{-100+\frac{i-1}{2}}}$
for $0\leq i\leq 401$, and radii $r_j=1+j$ for $1\leq j\leq199$. See
the right column of Figure \ref{F6} for example backprojection reconstructions of delta functions when $S$ is convex and non-convex. In the left column of Figure \ref{F6}, we show the artifacts due to Bolker as predicted by our theory. The predicted and observed artifacts match up well, and all artifact locations in the convex case lie outside the function support, which is in line with Theorem \ref{thm:RA-general}.  In the bottom-right of Figure \ref{F6}, there are additional circular shaped streaking artifacts which pass through the delta function. The circular streaks have centers at the end points of the red curve in the bottom-left of Figure \ref{F6}. These occur due to the sharp  cutoff at the boundary of the sinogram, since the circle centers are only finitely sampled. The boundary artifacts are less noticeable in the non-convex curve case, in the top-right of Figure \ref{F6}, and the artifacts due to Bolker appear more strongly.

\subsection{Phantom reconstructions}
Here, we present algebraic reconstructions of image phantoms from circular integral data. We consider two phantoms, one simple and one complex. The simple phantom consists of two rectangles with density 1, and the complex phantom is made up of a thin cross, a square, a hollow ellipse, and two circular phantoms, all of varying densities. The nonzero densities are arranged to fit within $D$, for both the convex and non-convex measurement curves considered. We present reconstructions using the Landweber method and TV regularization. Specifically, to implement TV, we find
\begin{equation}
\argmin_{\vx}\ (A\vx-\vb)^T(A\vx-\vb)+\alpha \sqrt{\|\nabla \vx\|^2_2+\beta^2},
\end{equation}
where $\alpha,\beta>0$ are regularization parameters. 
\begin{figure}[!h]
\centering
\begin{subfigure}{0.32\textwidth}
\includegraphics[width=0.9\linewidth, height=4cm, keepaspectratio]{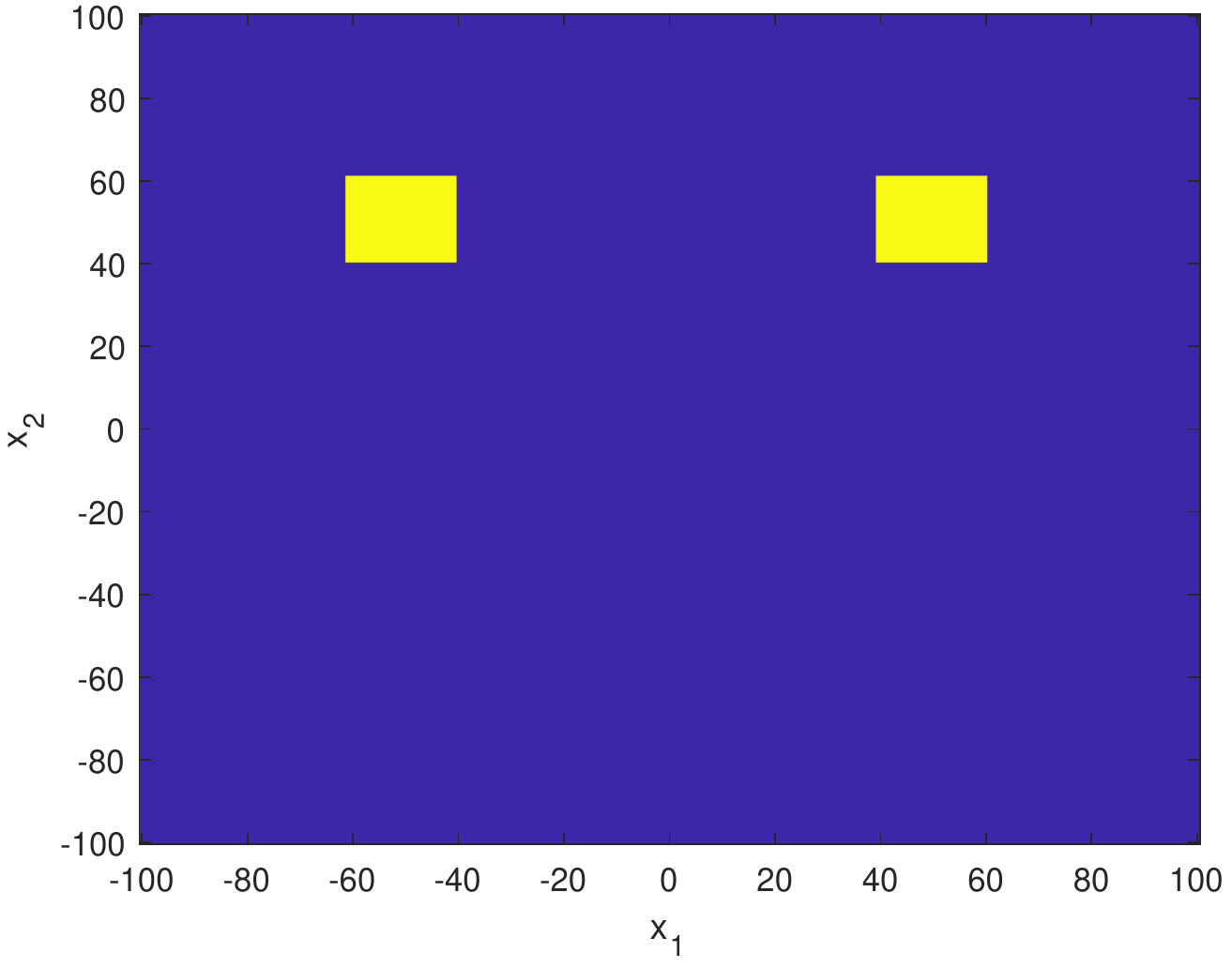}
\end{subfigure}
\begin{subfigure}{0.32\textwidth}
\includegraphics[width=0.9\linewidth, height=4cm, keepaspectratio]{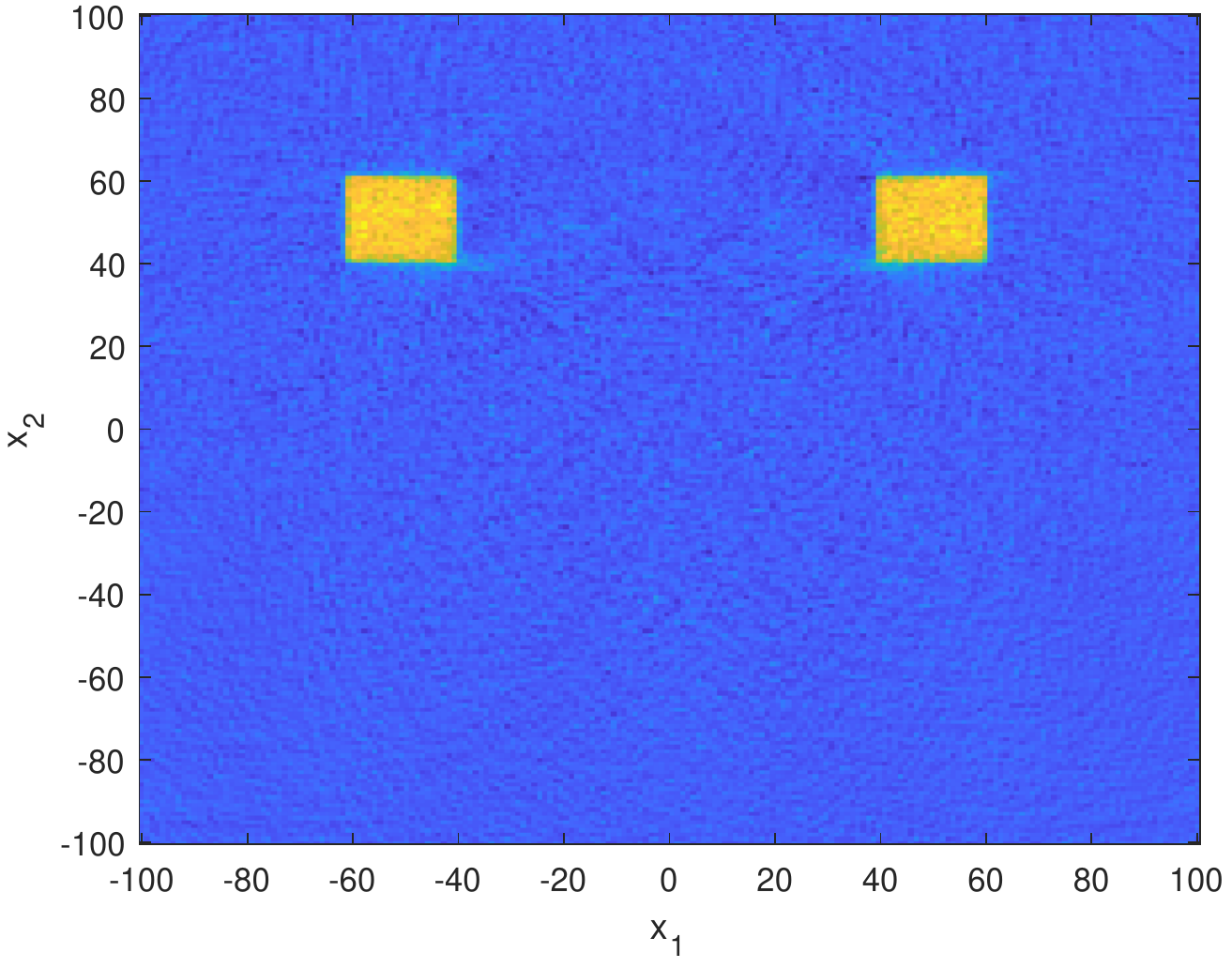} 
\end{subfigure}
\begin{subfigure}{0.32\textwidth}
\includegraphics[width=0.9\linewidth, height=4cm, keepaspectratio]{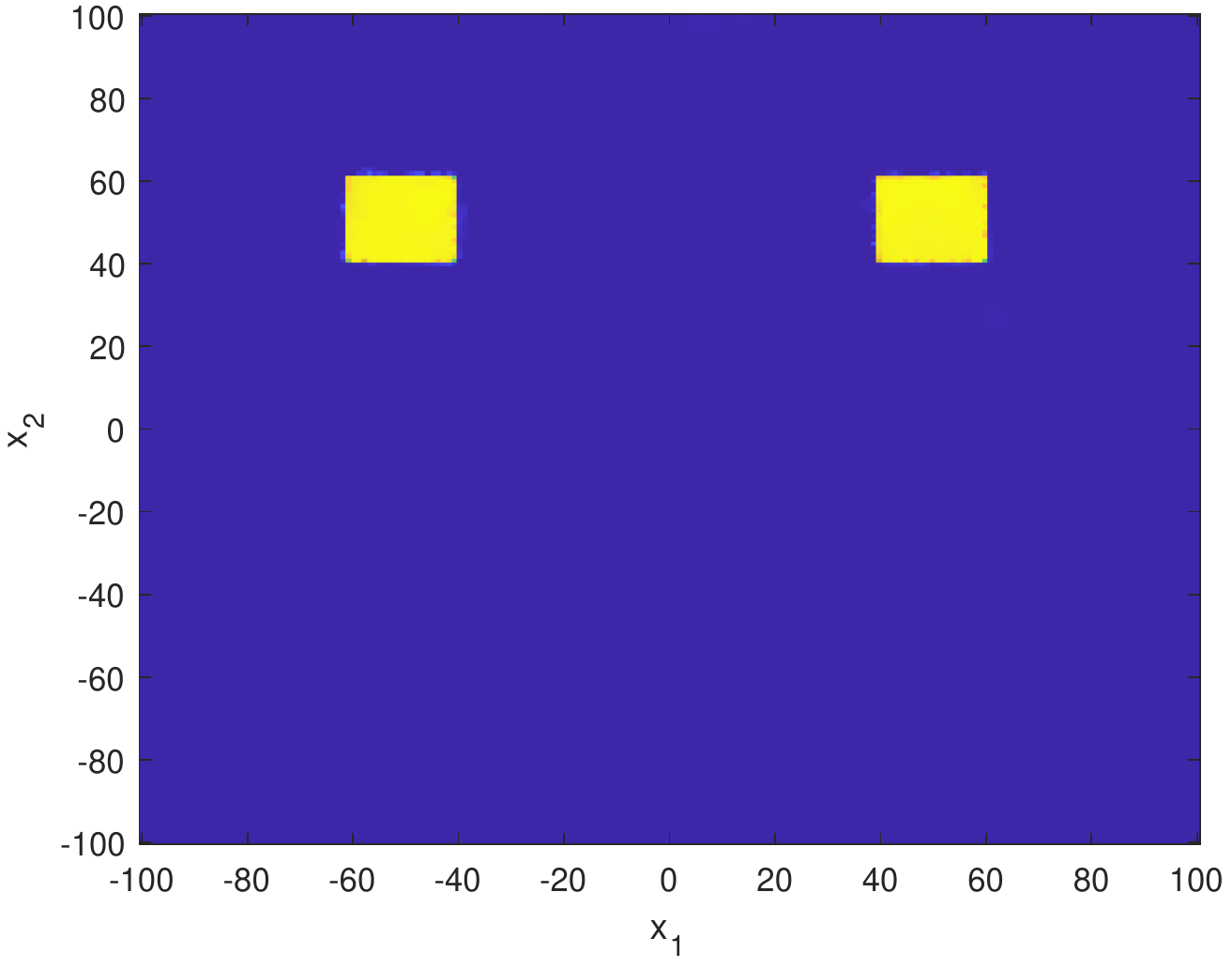}
\end{subfigure}
\begin{subfigure}{0.32\textwidth}
\includegraphics[width=0.9\linewidth, height=4cm, keepaspectratio]{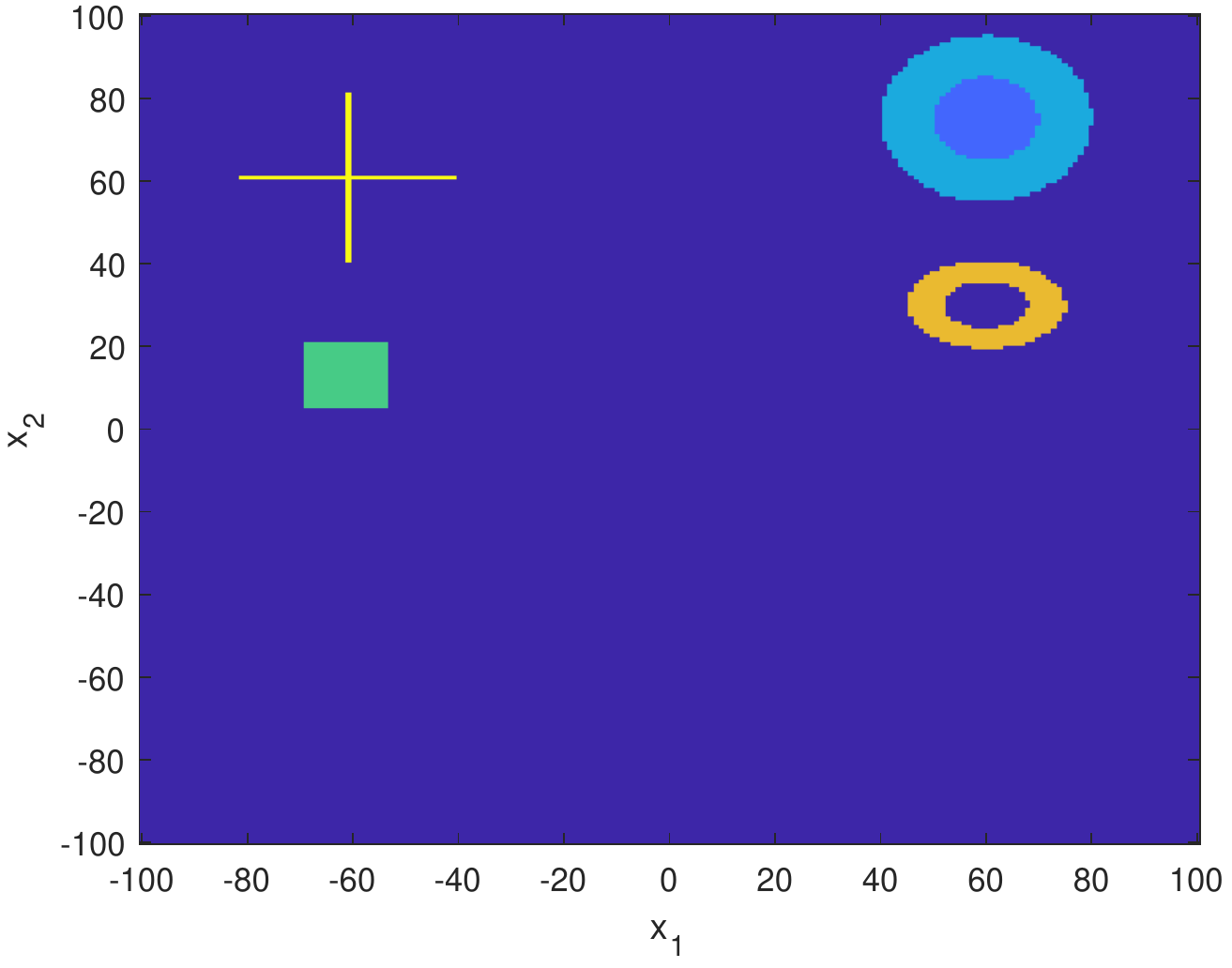}
\subcaption*{phantom} \label{F7d}
\end{subfigure}
\begin{subfigure}{0.32\textwidth}
\includegraphics[width=0.9\linewidth, height=4cm, keepaspectratio]{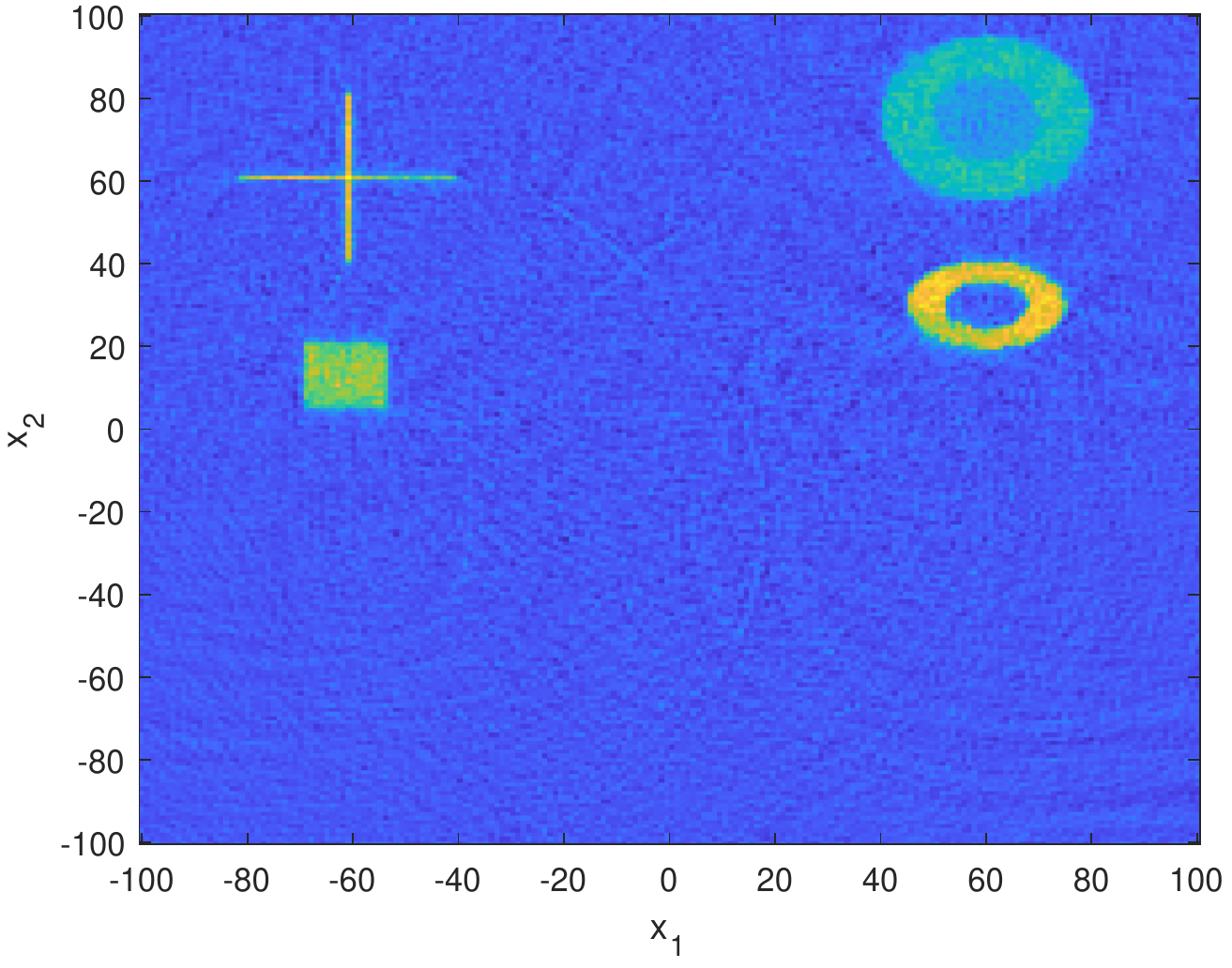} 
\subcaption*{Landweber} \label{F7e}
\end{subfigure}
\begin{subfigure}{0.32\textwidth}
\includegraphics[width=0.9\linewidth, height=4cm, keepaspectratio]{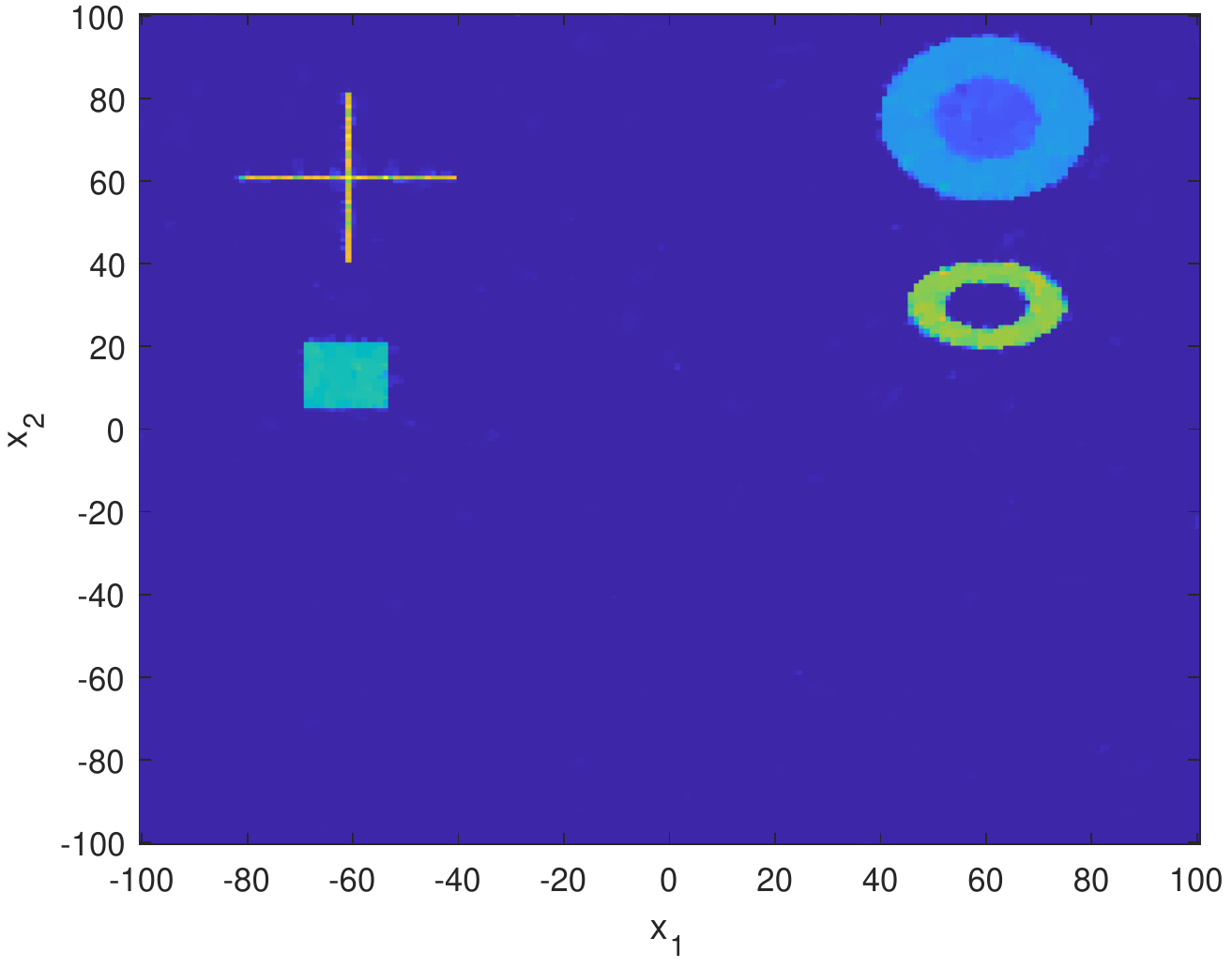}
\subcaption*{TV} \label{F7f}
\end{subfigure}
\caption{Reconstructions of image phantoms - non-convex curve.}
\label{F7}
\end{figure}

The data was simulated by $\vb=A_{\epsilon}\vx + \eta$, where
$A_{\epsilon}$ is a perturbed $A$. Specifically, to generate
$A_{\epsilon}$, the non-zero values of $A$ (i.e., the circular
integral weights) were multiplied by $1+u$, where $u \sim
U(-0.5,0.5)$, that is, $u$ is drawn from a uniform distribution on $[-0.5,0.5]$. We use $A_{\epsilon}$ to generate data to avoid inverse crime. The added noise is white noise drawn from a standard Gaussian $\eta \sim \mathcal{N}(0,\sigma)$, where $\sigma$ controls the noise level. The hyperparameters $\alpha,\beta$ were chosen using cross-validation, so there is no optimism in the results with respect to the selection of $\alpha,\beta$.
\begin{figure}[!h]
\centering
\begin{subfigure}{0.32\textwidth}
\includegraphics[width=0.9\linewidth, height=4cm, keepaspectratio]{phantom_1}
\end{subfigure}
\begin{subfigure}{0.32\textwidth}
\includegraphics[width=0.9\linewidth, height=4cm, keepaspectratio]{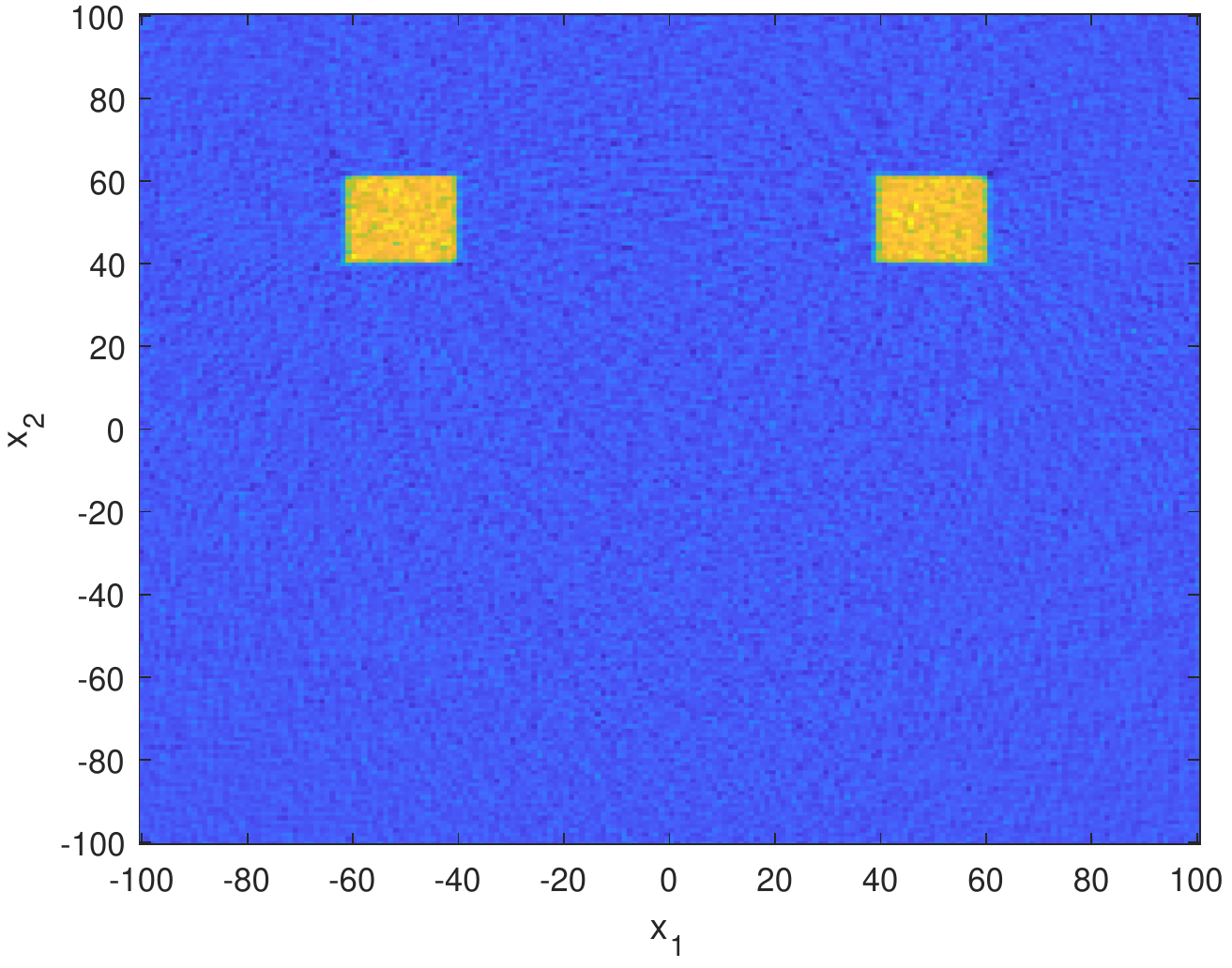} 
\end{subfigure}
\begin{subfigure}{0.32\textwidth}
\includegraphics[width=0.9\linewidth, height=4cm, keepaspectratio]{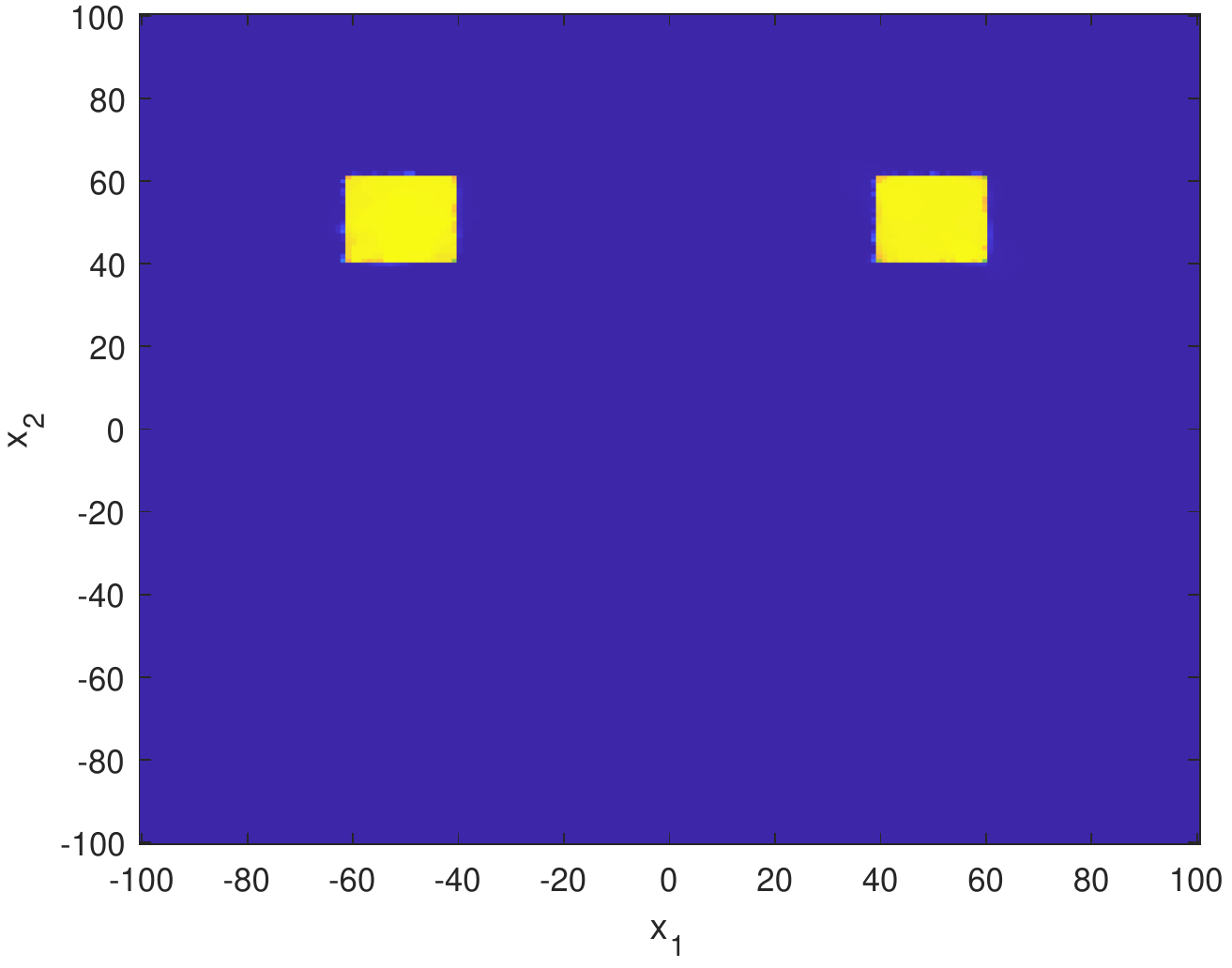}
\end{subfigure}
\begin{subfigure}{0.32\textwidth}
\includegraphics[width=0.9\linewidth, height=4cm, keepaspectratio]{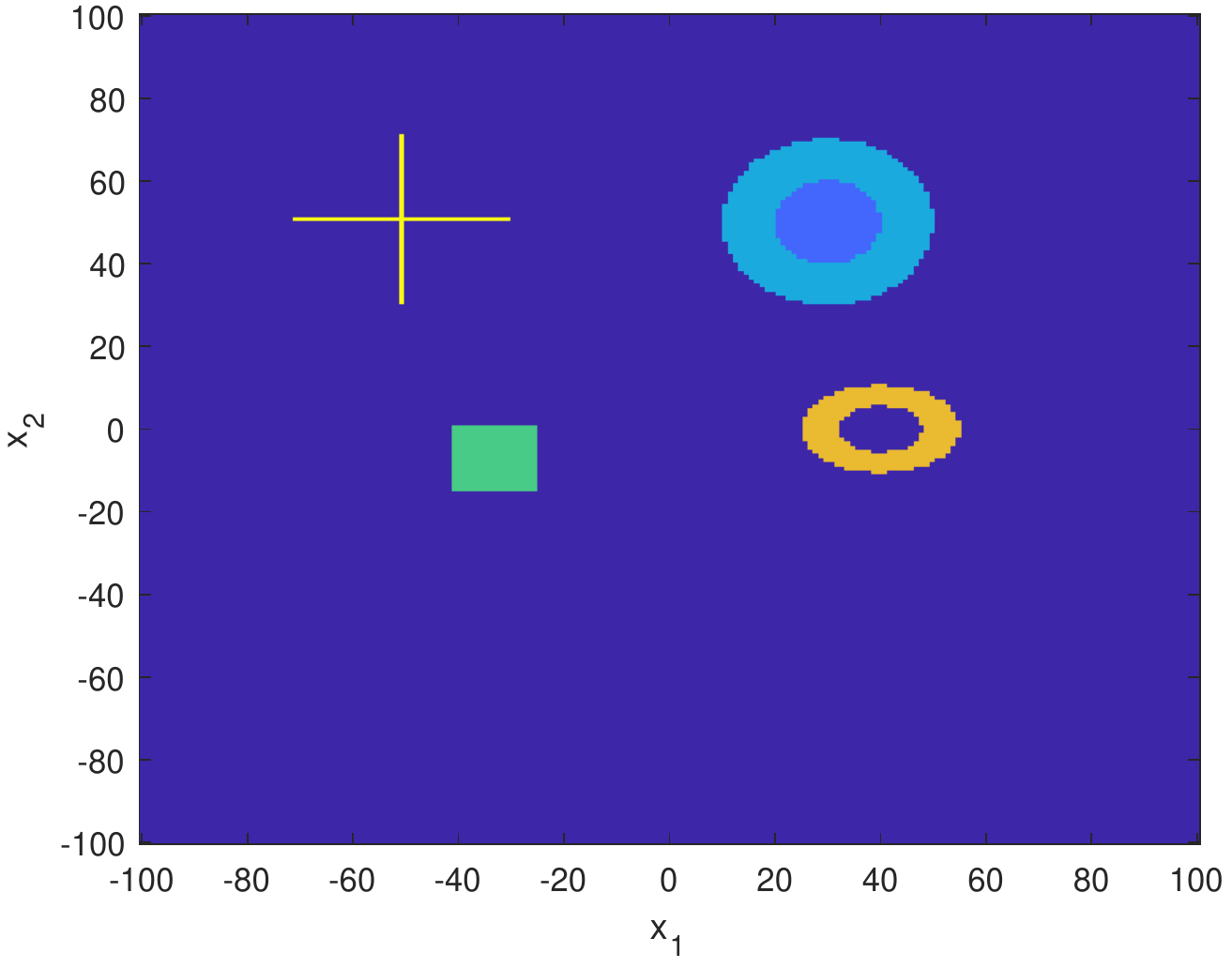}
\subcaption*{phantom} \label{F8d}
\end{subfigure}
\begin{subfigure}{0.32\textwidth}
\includegraphics[width=0.9\linewidth, height=4cm, keepaspectratio]{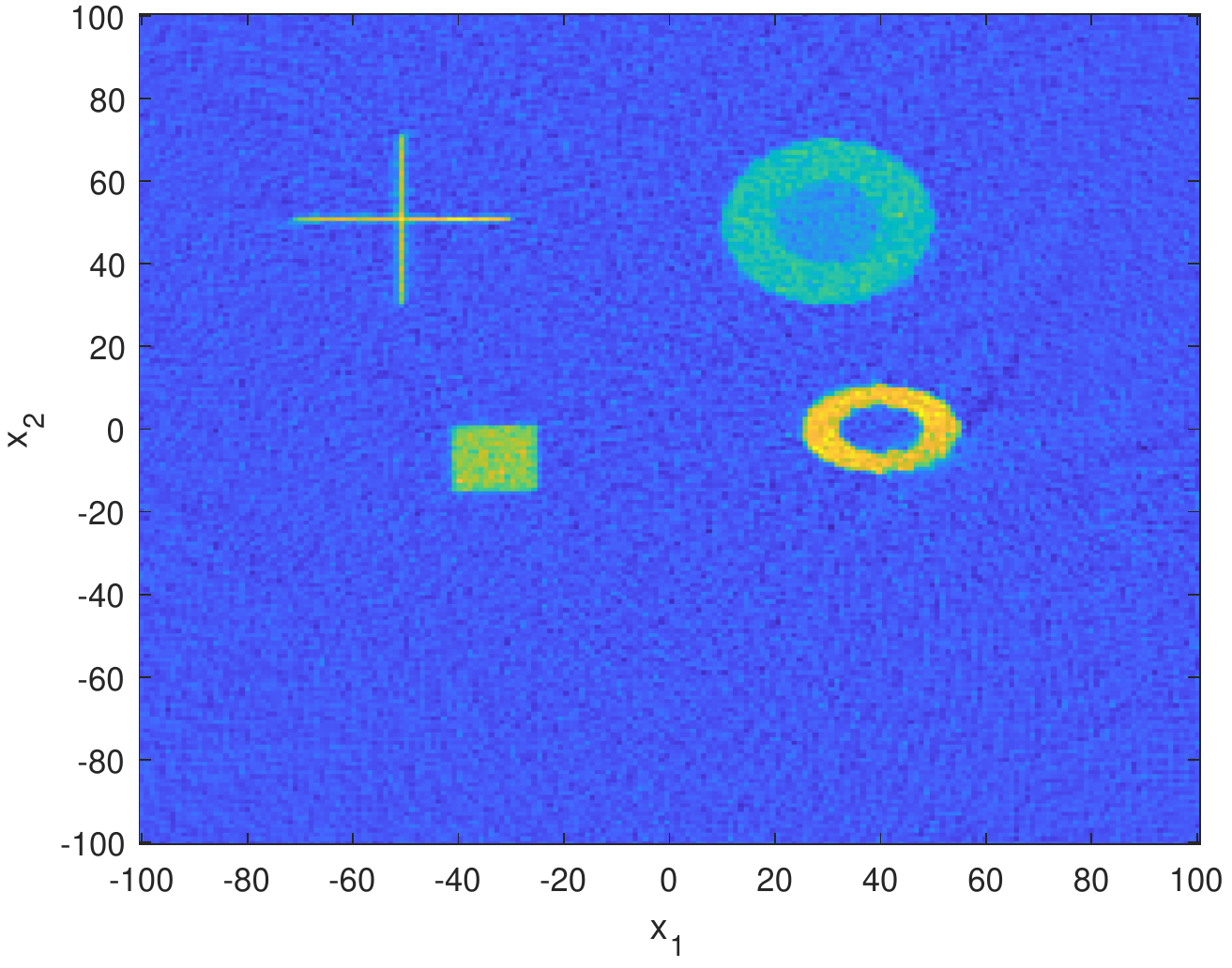} 
\subcaption*{Landweber} \label{F8e}
\end{subfigure}
\begin{subfigure}{0.32\textwidth}
\includegraphics[width=0.9\linewidth, height=4cm, keepaspectratio]{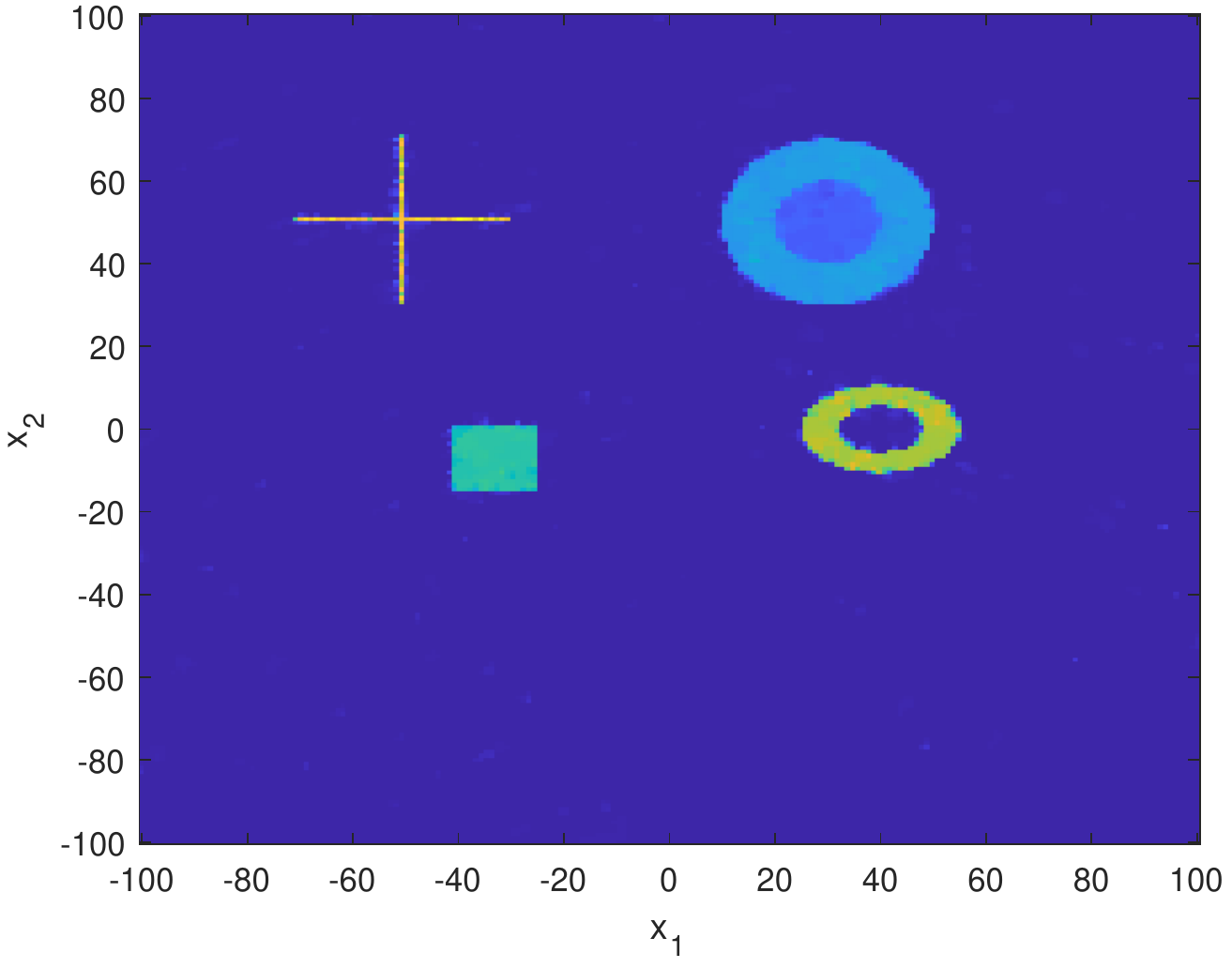}
\subcaption*{TV} \label{F8f}
\end{subfigure}
\caption{Reconstructions of image phantoms - convex curve.}
\label{F8}
\end{figure}

See Figure \ref{F7}, where we show reconstructions of the test phantoms using the Landweber method and TV in the non-convex curve case. We see severe artifacts in the Landweber reconstructions, which is not surprising given the inversion instabilities of $R$ when $S$ does not satisfy global convexity. The artifacts are largely suppressed in the TV reconstructions. TV  enforces sparse gradients, and thus it TV can combat additional, unwanted singularities in the reconstruction due to Bolker. In this example, TV is effective in removing the artifacts. 

In Figure \ref{F8}, we present image reconstruction of the simple and complex phantom in the convex curve case. The artifacts in the Landweber reconstruction of the simple phantom are minimal and we see only a background noise effect. In the Landweber reconstruction of the complex phantom, there are mild artifacts which appear as a streaking effect near the boundary of the hollow ellipse. In the TV reconstructions of both phantoms, the noise effects and streaking artifacts are suppressed.

\section{Conclusions and further work} In this paper, we presented
novel microlocal and injectivity results for a new generalized Radon
transform, $R$, which defines the integrals of square integrable
compactly supported functions over ellipsoids, hyperboloids, and
elliptic hyperboloids and generalizations with centers on a smooth surface, $S$. We showed
that $R$ was an FIO and proved that $R$ satisfied the Bolker condition
if condition \eqref{no tangents} holds, and this is if and only if
when $M$ is dimension zero.  We applied our theory to some
examples in section \ref{sect:examples}, and provided a more in-depth
analysis of a cylindrical scanning geometry of interest in URT, where
we proved injectivity results. Specifically, we showed that any $f \in
L^2_c$, with support in a cylindrical tube, could be reconstructed
uniquely from its integrals over spheroids with centers on a cylinder
which encapsulates the support of $f$. We also investigated the
visible singularities for the cylindrical geometry in section
\ref{visible-cylinder:sect}. In section \ref{recon:sect}, to validate
our microlocal theory and show the image artifacts, we presented image
reconstructions of delta functions and image phantoms. We also tested
a method of artifact reduction, using TV regularization and
cross-validation, which proved successful in the simulations
conducted.

In further work, we aim to investigate the potential practical applicability of the hyperboloid and elliptic hyperboloid Radon transform, as, in this work, we considered only applications of $R$ in fields such as URT, where the integral surfaces are spheroids. There is indication that the hyperboloid case may be of interest in proton therapy through the measurement of multiple gamma rays emitted as a cascade \cite{Panainoetal}.
We aim also to pursue three-dimensional image reconstruction methods for the cylindrical scanning geometry introduced in section \ref{sect:examples}. As evidenced in section \ref{visible-cylinder:sect}, the inverse of $R$ in the cylindrical case is severely unstable, as there are invisible singularities. Thus, it is likely that we will require strong regularization (e.g., TV or machine learning) to solve this problem. 

\section*{Acknowledgements:} The authors thank Plamen Stefanov for
helpful comments on this work. The third author's research was partially supported
by NSF grant 1712207 and Simons grant 708556. The first author wishes to acknowledge funding support from Brigham Ovarian Cancer Research Fund, The V Foundation, Abcam Inc., and Aspira Women’s Health.
Sean Holman was supported by the Engineering and Physical Sciences Research Council grant number EP/V007742/1.
\bibliographystyle{abbrv} 
\bibliography{RefEllipsoid_D1}

\end{document}